\documentclass[11pt,twoside, a4paper]{amsart}

\usepackage{mathptmx}
\usepackage{amscd}
\usepackage{amssymb}
\usepackage{eucal}
\usepackage{amsxtra}
\usepackage{verbatim}
\usepackage{url}
\usepackage{enumerate}
\usepackage[english]{babel}

\usepackage{comment}
\usepackage[all]{xy}

\DeclareFontEncoding{OT2}{}{} 
\DeclareTextFontCommand{\textcyr}{\fontencoding{OT2}
    \fontfamily{wncyr}\fontseries{m}\fontshape{n}\selectfont}

\def \A {{\mathbb A}}
\def \Q {{\mathbb Q}}
\def \Z {{\mathbb Z}}

\def \G {{\mathbb G}}
\def \P {{\mathbb P}}

\theoremstyle{plain}

\newtheorem{theo}{Theorem}[section]
\newtheorem{thm}[theo]{Theorem}
\newtheorem{prop}[theo]{Proposition}
\newtheorem{lem}[theo]{Lemma}
\newtheorem{corollary}[theo]{Corollary}
\newtheorem{defi}[theo]{Definition}

\numberwithin{equation}{section}

\newtheorem*{Def*}{Definition}

\def\beq{\begin{equation} \label}

\newtheorem*{claim*}{Claim}
\newtheorem*{exa*}{Example}
\newtheorem*{rem*}{Remark}
\newtheorem*{rems*}{Remarks}
\newtheorem*{fact*}{Fact}

\newcommand{\bthe}{\begin{thm}}
\newcommand{\ble}{\begin{lem}}
\newcommand{\bpr}{\begin{prop}}
\newcommand{\bco}{\begin{cor}}
\newcommand{\bde}{\begin{defi}}
\newcommand{\ethe}{\end{theo}}
\newcommand{\ele}{\end{lem}}
\newcommand{\epr}{\end{prop}}
\newcommand{\eco}{\end{cor}}
\newcommand{\ede}{\end{defi}}

\def\Res{{\rm Res}}

\def\Cor{{\rm Cor}}

\def \Gal {{\rm{Gal}}}

\def \Pic {{\rm {Pic}}}
\def \Div {{\rm {Div}}}

\def \Br {{\rm{Br}}}
\def \to {{\rightarrow}}
\def \Ker {{\rm Ker}}
\def \Hom {{\rm Hom}}

\DeclareFontFamily{U}{wncy}{}
\DeclareFontShape{U}{wncy}{m}{n}{%
<5>wncyr5%
<6>wncyr6%
<7>wncyr7%
<8>wncyr8%
<9>wncyr9%
<10>wncyr10%
<11>wncyr10%
<12>wncyr6%
<14>wncyr7%
<17>wncyr8%
<20>wncyr10%
<25>wncyr10}{} \DeclareMathAlphabet{\cyr}{U}{wncy}{m}{n}
\begin{document}

\title{On the equation $N_{K/k}(\Xi)=P(t)$}

\author{ Dasheng Wei}
\address{Academy of Mathematics and System Science, CAS, Beijing
  100190, P.R.China}

\email{dshwei@amss.ac.cn}

\date{\today}


\setcounter{section}{-1}

\begin{abstract}
For varieties given by an equation $N_{K/k}(\Xi)=P(t)$, where $N_{K/k}$ is the norm form
attached to a field extension $K/k$ and $P(t)$ in $k[t]$ is a nonzero polynomial, three topics have been investigated:
\begin{enumerate}[(i)]
\item computation of the unramified Brauer group of such varieties over arbitrary fields;

\item rational points and Brauer-Manin obstruction over number fields (under Schinzel's hypothesis);

\item zero-cycles and Brauer-Manin obstruction over number fields.
\end{enumerate}
In this paper, we produce new results in each of three directions. We obtain quite general results under the assumption that
$K/k$ is abelian (as opposed to cyclic in earlier investigation).
\end{abstract}

\maketitle{}

\section{Introduction} \label{sec.notation}

A class of geometrically integral varieties defined over a number
field $k$ satisfies the Hasse principle if a variety in this class
has a $k$-rational point as soon as it has rational points in all
the completions $k_v$ of the field $k$. For example, quadrics,
Severi-Brauer varieties are known to satisfy this principle.
However, counterexamples to the Hasse principle are also known even
in the class of geometrically rational varieties. In 1970, Manin
(\cite{Ma71,Ma86}) showed that an obstruction based on the
Brauer group of varieties, now referred to as the Brauer-Manin
obstruction, can often explain failures of the Hasse principle.
Further work (see \cite{CT92} for a survey) has shown that for some
classes of rational varieties the Brauer-Manin obstruction is the
only obstruction to the Hasse principle.

For some classes of varieties for which the Hasse principle has been
proved, weak approximation is also known: namely, given a variety
$X$ over a number field $k$, and given a $k_v$-rational point $P_v$
of $X$ for each $v$ in a finite set of places of $k$, one may find a
$k$-rational point on $X$ as close as one wishes to each $P_v$ (for
the $v$-adic topology). However, for more general unirational
varieties, counterexamples to weak approximation are known, and one
may define a Brauer-Manin obstruction to weak approximation
(\cite{CT/S76,CT/S77,CT/S87}) and ask whether it is
the only obstruction in the class of geometrically unirational
varieties.

For general varieties, it seems quite unreasonable to hope for such
statements. A more reasonable conjecture relates to zero-cycles of
degree $1$.  A variety defined over a field $k$ has a zero-cycle of
degree $1$ over $k$ if and only if the degrees of the finite field
extensions $K/k$ over which it acquires a rational point are
globally coprime. There also exist counterexamples to the Hasse
principle for zero-cycles of degree 1. Similarly, one may define a
Brauer-Manin obstruction and ask whether this is the only
obstruction for arbitrary smooth projective varieties over a number
field (\cite{CT99}).

The present paper focuses on varieties defined over the ground field
$k$ by an equation
\begin{equation}\label{equation:main}
N_{K/k}(\Xi)=P(t),
\end{equation}
where $K/k$ is a finite field extension, $N_{K/k}$ denotes the norm
map, $\Xi$ is  a ``variable'' in $K$ and $P(t)\in k[t]$ is a
nonconstant polynomial.

To compute the Brauer-Manin obstruction, one must compute the Brauer
group of a smooth projective model of the variety under study. That
group is sometimes referred to as the unramified Brauer group of the
variety.

In 2003, Colliot-Th\'el\`ene, Harari and
Skorobogatov (\cite{CHS1}) discussed the unramified Brauer group for these
varieties. In fact, they
defined a partial compactification of the smooth locus of these
varieties. For this partial compactification, they gave a formula
for its vertical Brauer group and the quotient of its Brauer group
by the vertical Brauer group. They pointed out that sometimes the
unramified Brauer group can be calculated by the formula. However,
it is still open how to determine the unramified Brauer group for
more general cases.

Let $X$ be the partial compactification of the smooth locus of the affine variety defined by (\ref{equation:main}) (see \cite{CHS1}), $X^c$ its smooth compactification. In Section 2, building upon the formula for $\Br(X)$ in \cite{CHS1}, we compute $\Br(X^c)$ in several new cases:
\begin{enumerate}[(i)]
\item In Section 2.1, for $P(t)$ irreducible and $K/k$ "general", we show $\Br(X^c)=\Br_0(X^c)$ (Theorem \ref{Br-1}).

\item In Section 2.2, for $P(t)$ irreducible and $K/k$ abelian, we show that the quotient $\Br(X)/\Br(X^c)$ is $2$-torsion, and we show that
the quotient is trivial in many cases (Theorem \ref{equal-X}).

\item In Section 2.3, motivated by the question in \cite{CHS1}, we give a formula for $\Br(X^c)/\Br_0(X^c)$ when $K/k$ is Galois and $P(t)$ has all roots in $k$, not necessarily distinct.
\end{enumerate}

In a series of earlier papers, rational points and zero-cycles of
degree 1 have been studied on  smooth  projective models of
varieties defined by an equation $N_{K/k}(\Xi)=P(t)$, and more
generally on varieties fibred over the projective line whose general
fibre is birationally a principal homogeneous space under a torus.

Under Schinzel's hypothesis (H), the question was first studied
by Colliot-Th\'el\`ene and Sansuc (\cite{CT/S82}). Further work under (H) is due to Colliot-Th\'el\`ene,
Serre,  Skorobogatov  and  Swinnerton-Dyer (\cite{SD95,Se94,CTSD94,CTSSD98}).

As explained in \cite{CTSD94,CTSSD98}, a device due to Salberger \cite{Sal88,Sal89}
enables one to transform
some of the conditional proofs for rational points obtained under (H)
into unconditional proofs for zero-cycles of degree 1.

In \cite{CTSSD98},  restrictions are made on the fibres : on the one hand
one assumes that the Hasse principle and weak approximation hold on all smooth  fibres,
on the other  hand one requires some abelianity condition for the splitting fields
associated to the components of the singular fibres.

{\it In this paper, we prove results of the above type in cases not
covered by \cite{CTSSD98}: either the Hasse principle need not hold
on the fibres, or the abelian splitting condition is not fulfilled.}

We consider three types of varieties over the ground field $k$:
\begin{enumerate}[(1)]
\item Varieties defined by an equation  $N_{K/k}(\Xi)=P(t)$, where $K/k$
is  an abelian extension and
$\cyr{X}^2_\omega(\widehat{T})_P=\cyr{X}^2_\omega(\widehat{T})$ (see
Section 1 for definition).

\item Varieties defined by an equation
$(x_1^2-ax_2^2)(y_1^2-by_2^2)(z_1^2-abz_2^2)=P(t)$ where $a,b\in
k^*$.

\item Varieties defined by an equation  $N_{K/k}(\Xi)=P(t)$ where $K/k$
has prime degree (non-Galois) and $P(t)$ satisfies condition (\ref{assumption:P}).
\end{enumerate}

In Section 3, we prove:
\begin{thm}
Let $V$ be the smooth locus of one of the above three varieties.
Assume Schinzel's hypothesis holds. Then the Brauer-Manin
obstruction to the Hasse principle and weak approximation for rational points is
the only obstruction for any smooth proper model of $V$.
\end{thm}

In Section 4, we prove:
\begin{thm}
Let $V$ be the smooth locus of one of the above three varieties. If
there is no Brauer-Manin obstruction to the existence of a
zero-cycle of degree 1 on a smooth proper model $V^c$ of $V$, then
there is a zero-cycle of degree 1 on $V^c$ (defined over $k$).
\end{thm}

\section{Some recollections from \cite{CHS1}, and some complements}\label{section:recollection}

Let us   recall   some results from \cite{CHS1}. Let $k$ be a field
of characteristic zero, $\bar k$ a fixed algebraic closure of $k$, $\Gamma_k$ the absolute Galois group of $k$.
 Let $P(t)\in k[t]$ be a nonzero polynomial. We
consider the affine variety over $k$ defined by
\begin{equation}\label{X} N_{K/k}(\Xi)=P(t)
\end{equation}
where $K/k$ is {\it a finite product of finite separable field extensions of k},
$\Xi$ is a "variable" in $K$
and $N_{K/k}$ is the  formal norm associated with
$K/k$. In \cite{CHS1} $K/k$ is a field extension, in fact much  of the theory holds
for $K/k$ a product of fields. This observation is used in subsequent sections (e.g., Lemma \ref{form} and Theorem \ref{rational-2}).

Let $V\subset \A_k^{n+1}\simeq \A_{k}^1\times R_{K/k}(\A_K^1)$ be
the {\it smooth locus} of the affine hypersurface defined by the
equation (\ref{X}). If $P(t)$ is separable, $V$ is the hypersurface defined by (\ref{X}). The projection $(\Xi,t) \mapsto t$ defines a
surjective morphism $p: V\rightarrow \A_k^1$. Let $U_0\subset
\A_k^1$ be the open subset defined by $P(t)\neq 0$ and
$U=p^{-1}(U_0)\subset V$.

Let $T=R_{K/k}^1(\G_m)$,
$\widehat T$ the $\Gamma_k$-module of characters of $T$ and $T^c$ a smooth
$T$-equivariant compactification (see \cite{CHS2}). The contracted
product $U\times^T T^c$ is a partial compactification of $U$ which is
proper and smooth over $U_0$. Let $X$ be the $k$-variety by
gluing $V$ and $U\times^T T^c$ along $U$, which is separated and smooth over $k$ . In this paper $X$ is
called the \emph{CTHS} partial compactification of $V$, which is first constructed in \cite{CHS1}.

For any $k$-variety $V$, we denote $\overline V:= V\times_k \bar k$. In \cite[Section 2]{CHS1} the following results are established : The natural
map $\bar k^\times \to \bar k[X]^\times$ is an isomorphism,
$\Pic(\overline X)$ is finitely generated and torsionfree,
and $\Br(\overline X)=0$.

By Hironaka's theorem (\cite{Hi}), one may find a smooth proper compactification $X^{c}$ which contains $X$ with a morphism $X^{c} \to \P^1_{k}$ extending the
morphism $X \to \A^1_{k}$. There are natural inclusions
$$\Br(X^{c}) \hookrightarrow  \Br(X).$$
One is interested in knowing when these are equalities.

There are also natural inclusions
$$\Br_{vert}(X) \hookrightarrow  \Br(X)$$
and
$$\Br_{vert}(X^{c})\hookrightarrow  \Br_{vert}(X^{c})$$
where $\Br_{vert}(X)$ denotes the subgroup of  $\Br(X)$ whose image
in the Brauer group $\Br(X_{\eta})$ of the generic fibre of $X \to
\P^1_k$ is in the image of $\Br(k(\P^1)) \to \Br(X_{\eta})$, and
similarly for $\Br_{vert}(X^{c})$. One is interested in deciding which groups contribute to $\Br(X)$
and to $\Br(X^{c})$, and ultimately in computing $\Br(X^{c})$.

Let $F/k$ be a Galois extension and $M$ a
$\Gal(F/k)$-module. For
simplicity, in this paper we always denote
$$H^i(F/k,M):= H^i(\Gal(F/k),M).$$
If $F=\bar k$, we denote it simply by $H^i(k,M)$. We define
$$\cyr{X}^i_{\omega}(F/k,M):=\bigcap_{g\in \Gal(F/k)}\Ker[H^i(F/k,M)\rightarrow H^i(\left<g\right>,M)]$$
where $\left<g\right>$ is the  cyclic subgroup of $\Gal(F/k)$ generated
by $g$. If $F=\bar k$, we also simply denote
$\cyr{X}^i_{\omega}(\bar k/k,M)$ by $\cyr{X}^i_{\omega}(M)$.

We denote $\Z[L/k]$ to be the $\Gamma_k$-module
$\Z[\Gamma_k/\Gamma_L]$ for any finite field extension $L/k$. Write $P(t)=cp_1(t)^{e_1}\cdots
p_m(t)^{e_m}$ where $p_i(t)$ is monic and irreducible. Denote by $\Z_P$ the
permutation $\Gamma_k$-module associated with the polynomial
$P(t)$, which is a direct sum of $\Z[L_i/k]$ where
$L_i=k[t]/(p_i(t))$. Let $N_i=N_{L_i/k}\in \Z[L_i/k]$ be
$\sum_{\sigma} \sigma$ with $\sigma$ running over all embeddings of
$L_i$ into $\bar k$.

Let $j_P: \Z\rightarrow \Z_P$ be defined by sending $1$ to
$-(e_1N_1+\cdots+e_mN_m)$.
Let $F/k$ be a Galois extension which splits $P(t)$. For any $\Gal(F/k)$-module $M$, $j_P$ induces a
morphism $M\rightarrow M\otimes\Z_P$. We define
$$\cyr{X}^2_\omega(F/k,M)_P:=\Ker[\cyr{X}^2_\omega(F/k,M)\rightarrow \cyr{X}^2_\omega(F/k,M\otimes\Z_P)].$$
If $F=\bar k$, we simply denote it by $\cyr{X}^2_\omega(M)_P$.

With the notation as above,  Colliot-Th\'el\`ene, Harari and Skorobogatov proved the following theorem:
\begin{thm}\label{CHS} \cite[Proposition 2.5]{CHS1} Let $X/k$ be the CTHS partial compactification of $V$ as above. Then:
\begin{enumerate}[(i)]
\item The following sequence is exact
\begin{equation}\label{bs} 0  \rightarrow H^1(k,\widehat T\otimes
\Z_P)/j_{P*}H^1(k,\widehat T) \rightarrow H^1(k,\Pic(\overline X))
\rightarrow \cyr{X}^2_\omega (\widehat T)_P \rightarrow 0
.\end{equation}
\item The elements of $\Br(X)$ whose image in
$H^1(k,\Pic(\overline X))$ come from $H^1(k,\widehat T\otimes \Z_P)$
are precisely the elements of $\Br_{vert}(X)$.
\end{enumerate}
\end{thm}

Under the assumption
$H^3(k,\bar k^\times)=0$, which is satisfied if $k$ is a number
field, the above exact sequence (\ref{bs}) identifies with
\begin{equation}\label{bs2} 0
\to \Br_{vert} (X)/  \Br_{0}(X) \to \Br(X)/ \Br_{0}(X) \to
\cyr{X}^2_\omega(\widehat T)_P \rightarrow 0,\end{equation}
where $\Br_{0}(X)$ denotes the image of the natural map $\Br(k) \to
\Br(X)$.

Let us mention an easy application of Theorem \ref{CHS}.
\begin{prop}\label{abel}
Let $k, K, P(t), X$ be as above. Assume that $K$ is a field, and
that the extension $K/k$ is abelian, and that $P(t)$ is irreducible.
Then
\begin{enumerate}[(a)]
\item $H^1(k,\widehat T\otimes \Z_P)/j_{P*}H^1(k,\widehat T)=0$.
\item $ \Br_{vert}(X)\simeq \Br_{0}(X)$.
\item  If $H^3(k,{\overline k}^\times)=0$, then $\Br(X)/\Br_{0}(X) \simeq \cyr{X}^2_\omega(\widehat T)_P$.
\item If   $K/k$ is cyclic, then $H^1(k,\Pic(\overline{X}))=0$, hence $\Br(X)=\Br_{0}(X)$.
\end{enumerate}
\end{prop}
\begin{proof} Let $L=k[t]/(P(t))$.
 We know
 $$\begin{aligned}H^1(k,\widehat T\otimes \Z_P)&\simeq H^1(L,\widehat T)\simeq\Ker[H^1(L,\Q/\Z)\rightarrow H^1(L.K,\Q/\Z)]\\&\simeq\Hom(\Gal(L.K/L),\Q/\Z).\end{aligned}$$ Since the extension $K/k$ is abelian,  the restriction map $$H^1(K/k,\Q/\Z)\rightarrow H^1(L.K/L,\Q/\Z) \text{ is surjective},$$ hence $H^1(k,\widehat T\otimes \Z_P)/j_{P*}H^1(k,\widehat T)=0$, $(b) \text{ and } (c)$ follows from the sequences (\ref{bs}) and (\ref{bs2}).

If $K/k$ is cyclic, since $\widehat T$ is split by $K$, then
$$ \cyr{X}^2_\omega(\widehat T)=\cyr{X}^2_\omega(K/k,\widehat T)=\bigcap_{g\in \Gal(K/k)}\Ker[H^2(K/k,M)\rightarrow H^2(\left<g\right>,M)]=0.$$
By the sequence (\ref{bs}), we have $H^1(k,\Pic(\overline{X}))=0$.
\end{proof}

\begin{prop} \label{compact-omega}
Let $k, K, P(t),L_i, X \text{ and }X^c$ be as above. The natural surjective map $\Pic(\overline{X^c}) \to \Pic(\overline{X})$ induces an injection
$$H^1(k,\Pic(\overline{X^c})) \hookrightarrow H^1(k,\Pic(\overline{X}))$$
and isomorphisms
$$\cyr{X}^1_\omega (\Pic(\overline {X^c}))\simeq \cyr{X}^1_\omega
(\Pic(\overline {X})) \simeq \cyr{X}^2_\omega(\widehat T'),$$
where $T'$ is the torus over $k$ defined by
$$\prod_{i=1}^m N_{L_i/k}(\Upsilon_i)^{e_{i}}\cdot N_{K/k}(\Xi)=1.$$
\end{prop}
\begin{proof}Let $\Div_{\overline {X^c}\setminus \overline X}(\overline {X^c})$ be the group of divisors of $\overline {X^c}$ with supports in $\overline {X^c}\setminus \overline X$.
Since $\bar k[X]^\times=\bar k^\times$, we have the following
exact sequence
$$0\rightarrow \Div_{\overline {X^c}\setminus \overline X}(\overline {X^c})\rightarrow \Pic(\overline {X^c})\rightarrow \Pic(\overline X)\rightarrow 0,$$
then the injectivity of $$H^1(k,\Pic(\overline{X^c})) \rightarrow H^1(k,\Pic(\overline{X}))$$ is implied by $H^1(k,\Div_{\overline {X^c}\setminus \overline X}(\overline {X^c}))=0$, which is obvious since $\Div_{\overline
{X^c}\setminus \overline U}(\overline {X^c})$ is a
permutation $\Gamma_k$-module.

Let $U (\subset X \subset X^{c})$
be the smooth affine  variety defined by
$$N_{K/k}(\Xi)=P(t) \neq 0.$$
Let $\overline U_0\subset \A_k^1$ be a open subset defined $P(t) \neq 0$. There exists an isomorphism of
varieties over $\bar k$
\begin{equation}\label{U-iso}
\overline U\simeq \overline U_0 \times \overline T\simeq
\overline U_0 \times \G_m^{[K:k]-1},
\end{equation}
hence $\Pic(\overline U)=0$. Therefore
we have the following
exact sequence
$$0\rightarrow \bar k[U]^\times/\bar k^\times \rightarrow \Div_{\overline {X^c}\setminus \overline U}(\overline {X^c})\rightarrow \Pic(\overline {X^c})\rightarrow 0.$$
Since $\Div_{\overline
{X^c}\setminus \overline U}(\overline {X^c})$ is a
permutation $\Gamma_k$-module, we have $H^1(k, \Div_{\overline
{X^c}\setminus \overline U}(\overline {X^c}))=0$. Then we have the
following exact sequence
$$0\rightarrow H^1(k,\Pic(\overline{X^c}))\rightarrow H^2(k,\bar k[U]^\times/\bar k^\times)\rightarrow H^2(k,\Div_{\overline {X^c}\setminus \overline
U}(\overline {X^c})).$$ Thus we obtain the exact sequence
$$0\rightarrow \cyr{X}_\omega^1(\Pic(\overline{X^c}))\rightarrow \cyr{X}^2_\omega(\bar k[U]^\times/\bar k^\times)\rightarrow \cyr{X}^2_\omega(\Div_{\overline {X^c}\setminus \overline
U}(\overline {X^c})).$$ Since $\Div_{\overline {X^c}\setminus
\overline U}(\overline {X^c})$ is a permutation $\Gamma_k$-module, we have
 $$\cyr{X}^2_\omega(\Div_{\overline {X^c}\setminus \overline U}(\overline {X^c}))=0.$$
Thus
$\cyr{X}_\omega^1(\Pic(\overline{X^c}))\simeq \cyr{X}^2_\omega(\bar
k[U]^\times/\bar k^\times)$.

Since $\bar k[X]^\times=\bar k^\times$, we also have the exact sequence
$$0\rightarrow \bar k[U]^\times/\bar k^\times \rightarrow \Div_{\overline
{X}\setminus \overline U}(\overline {X})\rightarrow \Pic(\overline
{X})\rightarrow 0.$$ By similar
arguments as above, we also have $\cyr{X}^1_\omega(\Pic(\overline{X}))
\simeq \cyr{X}^2_\omega(\bar
k[U]^\times/\bar k^\times)$.

By (\ref{U-iso}), we know
$\bar k[U]^\times/\bar k^\times$ is generated by the linear factors of $P(t)$
over $\bar k$ and all regular fuctions over $\bar k$ of the torus $R_{K/k}(\G_m)$, hence $\bar k[U]^\times/\bar k^\times\simeq \widehat T'$ as
$\Gamma_k$-modules, which implies
$\cyr{X}^2_\omega(\bar k[U]^\times/\bar k^\times)\simeq
\cyr{X}^2_\omega(\widehat T')$.
\end{proof}

\section{Calculation of the Brauer group}

In the remainder of this paper we always assume $K/k$ is a field extension.  In this section, we keep notations as in \cite{CHS1} and as in Section 1. In \cite{CHS1}, Colliot-Th\'el\`ene, Harari and Skorobogatov
gave a formula for the vertical Brauer group $\Br_{vert}(X)$ and the
quotient $\Br(X)/\Br_{vert}(X)$. They pointed out that the vertical
unramified Brauer group $\Br_{vert}(X^c)$ can be calculated by the
formula. However, it is still open how to determine the unramified
Brauer group $\Br(X^c)$. The aim of this section is to investigate
the latter group.

\subsection{The case $P(t)$ is irreducible and some linear independence
condition is satisfied}

\begin{lem} \label{Sha} Let $P(t)$ be an irreducible polynomial and $L=k[t]/(P(t))$.
Let $K^{cl}$ (resp. $L^{cl}$) be the Galois closure of $K$ (resp.
$L$) over $k$. If $L\cap K^{cl}=k$, then $\cyr{X}^2_\omega(\widehat
T)_P=0$.
\end{lem}
\begin{proof}
Let $E$ be the compositum $K^{cl}.L^{cl}$. We can see $\cyr{X}^2_\omega(\widehat
T)=\cyr{X}^2_\omega(E/k,\widehat T)$ since $\widehat T$ is split by $E$. Then
$$\aligned \cyr{X}^2_\omega(\widehat T)_P&= \Ker[j_{P*}:\cyr{X}^2_\omega(\widehat
T)\rightarrow\cyr{X}^2_\omega(\widehat T \otimes \Z_P)]\\
&=\Ker[j_{P*}:\cyr{X}^2_\omega(E/k,\widehat
T)\rightarrow\cyr{X}^2_\omega(E/k,\widehat T \otimes \Z_P)]\\
&= \Ker[j_{P*}:\cyr{X}^2_\omega(E/k,\widehat T)\rightarrow
H^2(E/k,\widehat T \otimes \Z_P)]. \endaligned$$
Using Shapiro's lemma,
we have
$$\cyr{X}^2_\omega(\widehat
T)_P=\Ker[\Res_{k/L}:\cyr{X}^2_\omega(E/k,\widehat T)\rightarrow
H^2(E/L, \widehat T)].$$ We have the following commutative diagram
\begin{equation*}
\xymatrix @R=15pt @C=15pt{ &\cyr{X}^2_\omega(E/k,\widehat T)\ar[r] &
H^2(E/L,\widehat T)\\
&\cyr{X}^2_\omega(K^{cl}/k,\widehat T) \ar[r] \ar[u]^{f} &
H^2(L.K^{cl}/L,\widehat T)\ar[u]^{g}. }
\end{equation*}
Since $\widehat T$ is also split by $K^{cl}$, $f$ is an
isomorphism. By the Hochschild-Serre spectral sequence, one obtains
the exact sequence $$ H^1(E/L.K^{cl},\widehat
T)^{\Gal(E/L)}\rightarrow H^2(L.K^{cl}/L, \widehat T)\rightarrow
H^2(E/L,\widehat T).$$ Obviously
$H^1(E/L.K^{cl},\widehat T)=0$, hence $g$ is injective. So we
have
$$\cyr{X}^2_\omega(\widehat T)_P=\Ker[\cyr{X}^2_\omega(K^{cl}/k,\widehat
T) \rightarrow H^2(L.K^{cl}/L,\widehat T)].$$ Since $L\cap
K^{cl}=k$, we have
$$H^2(K^{cl}/k,\widehat T)\simeq H^2(L.K^{cl}/L, \widehat T),$$ hence
 $\cyr{X}^2_\omega(\widehat T)_P=0.$
\end{proof}


\begin{thm} \label{Br-1} Let $P(t)$ be an irreducible polynomial over $k$ and $L=k[t]/(P(t))$. Let
$K^{cl}$ be the Galois closure of the field extension $K/k$. Let $X/k$ be the CTHS partial compactification  for $P(t)$ and $K$ as in Section 1.

If $L\cap K^{cl}=k$, then $H^1(k,\Pic(\overline X))=0$, hence
$\Br(X)=\Br_0(X)$.
\end{thm}
\begin{proof} By the injectivity of
$\Br(X)/\Br_0(X)\rightarrow H^1(k,\Pic(\overline X))$, we only
need to show $H^1(k,\Pic(\overline X))=0$. Since $L\cap K^{cl}=k$, we have $\cyr{X}^2_\omega(\widehat T)_P=0$
by Lemma \ref{Sha}. By the sequence (\ref{bs}), we only need to show $H^1(k,\widehat T\otimes
\Z_P)/j_{P*}H^1(k,\widehat T)=0$.


By Shapiro's lemma and the assumption $L\cap K^{cl}=k$, we have the
following commutative diagram

 $$\xymatrix{
  H^1(k,\widehat T) \ar[dd]^{\cong}\ar[dr]_{\Res_{k/L}} \ar[r]^(0.45){j_{P*}}
                &H^1(k,\widehat{T}\otimes \Z_P) \ar[d]^{\cong}  \\
                &H^1(L,\widehat T) \ar[d]^{\cong}           \\
    H^1(K^{cl}/k,\widehat T)\ar[r]^{\cong}            &H^1(L.K^{cl}/L,\widehat T),}$$
which implies
\begin{equation*}
H^1(E/k,\widehat T\otimes
\Z_P)/j_{P*}H^1(E/k,\widehat T)=0.
\end{equation*}
\end{proof}

By the above theorem and Proposition \ref{compact-omega}, we have an
application for certain  multi-norm tori.

\begin{corollary}\label{Sha-T} Let $L$ and $K$ be field extensions over $k$. Let $T'$ be the torus over $k$ defined by $N_{L/k}(\Upsilon)N_{K/k}(\Xi)=1$.
If $L\cap K^{cl}=k$ or $K/k$ is cyclic, then
$\cyr{X}^2_\omega(\widehat T')=0$, in particular, if $k$ is a number
field, principal homogeneous spaces of $T'$ satisfy the Hasse
principle and weak approximation.
\end{corollary}
\begin{proof} Let $P(t)$ be an irreducible polynomial over $k$ such that $L\simeq k[t]/(P(t))$.
Let $X$ be the \emph{CTHS} partial compactification (see Section 1) for $P(t)$ and $K$. By Proposition \ref{compact-omega}, we have
$\cyr{X}^2_\omega(\widehat T')\subset H^1(k,\Pic(\overline {X}))$.
The statement follows from Proposition \ref{abel} for $K/k$ cyclic and Theorem
\ref{Br-1} for $L\cap K^{cl}=k$.

Suppose $k$ is a number field. Since $\cyr{X}^2_\omega(\widehat
T')=0$, the result follows from the fact that the Brauer-Manin
obstruction is the only one to the Hasse principle and
weak approximation for principal homogeneous spaces of tori (\cite[Theorem 8.12]{San} or \cite[Theorem 5.2.1]{Sko01}).
\end{proof}

\begin{rem*} The case that $K/k$ is cyclic in Corollary \ref{Sha-T} was proved
in an unpublished paper of Sansuc. If  $L^{cl}\cap K^{cl}=k$, this corollary was also proved by  Pollio and Rapinchuk (see \cite{PR}), and a more general result was proved in \cite{DW}.
\end{rem*}

\subsection{The case $P(t)$  irreducible and $K/k$ an abelian field  extension}

In this section, we investigate the case that $P(t)$ is irreducible and $K/k$ abelian. The main result is Theorem \ref{equal-X}, which in many situations gives good control of the quotient $\Br(X)/\Br(X^c)$.

\begin{lem} \label{form} Let $P(t)$ be an irreducible
polynomial over $k$, $K/k$ an abelian extension, and $X$ the CTHS partial compactification in Section 1.
Denote by $K^g$ the fixed field of $K$ by $g\in \Gamma_k$,
$L=k[t]/(P(t))$ and $L_g=L\cap K^g$.

Then any element of $\Br(X_{K^g})$ which comes from $\Br(X)$ has the form
$$\rho+\sum_{\sigma\in \Gal(L_g/k)}\psi(\sigma)(p(t)^{\sigma},\chi),$$ where
$\rho\in \Br(K^g)$, $p(t)$ is a fixed irreducible factor of $P(t)$ over $L_g$, $\chi$ is a fixed primitive character of $\Gal(K/K^g)$, $\psi\in \Hom(\Gal(L_g/k),\Z/b\Z)$ with
$b:=[K:K\cap (L.K^g)]$.
\end{lem}
\begin{proof}   The proof is based on Theorem \ref{CHS}, and the observation that $\Br(X_{K^g})=Br_{vert}(X_{K^g})$.
Since
$K/K^g$ is cyclic, we have
$$\cyr{X}^2_\omega (K^g,\widehat T)=\cyr{X}^2_\omega
(K/K^g,\widehat T)=0.$$ By the sequence (\ref{bs}), we have the commutative diagram:
$$
\xymatrix @R=12 pt{ 0 \ar[r] & H^1(k,\widehat{T}\otimes\Z_P)/j_{P*}H^1(k,\widehat T) \ar[d]\ar[r] &
H^1(k,\Pic(\overline X)) \ar[d]^{\Res}\\
0 \ar[r] & H^1(K^g,\widehat T\otimes
\Z_P)/j_{P*}H^1(K^g,\widehat T)
\ar[r]^(.62){\cong} & H^1(K^g,\Pic(\overline X)).}
$$

We want to investigate the restriction of $\Br(X)$ to $\Br(X_{K^g})$, $i.e.$, the restriction of $H^1(k,\Pic(\overline X))$ to $H^1(K^g,\Pic(\overline X))$. Since the image of $H^1(k,\Pic(\overline X))$ is $\Gamma_k$-invariant, we only need consider $[H^1(K^g,\widehat T\otimes \Z_P)/j_{P*}H^1(K^g,\widehat T)]^{\Gamma_k/\Gamma_{K^g}}$ by the above commutative diagram, noting that
$\Gamma_{K^g}$ is a normal subgroup of $\Gamma_k$ since $K/k$ is abelian.

We recall $L_g=L\cap K^g$. Write $P(t)=cp_1(t)\cdots
p_m(t)$, where $c\in k^\times$ and $p_i(t)\in L_g[t]$ monic and irreducible. Since $L_g\subset K$ and $K/k$ abelian, $L_g/k$ is Galois (abelian). Denote $\Delta:=\Gal(L_g/k)$, then $\Delta$ acts transitively on $\{p_i(t):1\leq i\leq m\}$. We fix $p(t)$ to be some $p_i(t)$ and let $L'=L_g[t]/(p(t))$. Therefore
$L\cong L'$ over $k$ and $$\Z_P\simeq \Z[L.K^g/K^g]\otimes \Z[\Delta]$$ as $\Gamma_k$-modules, which implies the natural isomorphism
$$H^1(K^g,\widehat T\otimes \Z_P)\simeq H^1(K^g,\widehat T\otimes \Z[L.K^g/K^g])\otimes \Z[\Delta]
 \simeq
H^2(L.K/L.K^g,\Z)\otimes \Z[\Delta].
$$
Denote
\begin{equation} \label{dif-R}R:=H^1(L.K/L.K^g,\Q/\Z) \simeq H^1(K/K\cap(L.K^g),\Q/\Z),\end{equation}
then $H^1(K^g,\widehat T\otimes \Z_P)\simeq R[\Delta]$.

The following diagram
$$\xymatrix{
  H^1(K^g,\widehat T) \ar[d]^{\cong} \ar[r]^{j_{P*}}
                &H^1(K^g,\widehat{T}\otimes \Z_P) \ar[d]^{\cong}  \\
   H^1(K/K^g,\Q/\Z)     \ar[r]^{\text{Res}}        &R[\Delta]}           $$
is commutative, where the map $\text{Res}$ is the diagonal map by restriction to $R$. Since $K/k$ is abelian, the image of $\text{Res}$ in $R[\Delta]$ is $R\cdot N$ and $R$ is a constant $\Gamma_k$-module, where $N=\sum_{\sigma\in \Delta}\sigma\in R[\Delta]$.
Then we have an isomorphism as $\Gamma_k$-modules
$$H^1(K^g,\widehat T\otimes \Z_P)/j_{P*}H^1(K^g,\widehat
T)\simeq R[\Delta]/(R\cdot N).$$ Then \begin{equation*} [H^1(K^g,\widehat T\otimes
\Z_P)/j_{P*}H^1(K^g,\widehat T)]^{\Gamma_k/\Gamma_{K^g}}\simeq (R[\Delta]/(R\cdot
N))^\Delta.\end{equation*}
The long exact sequence in Galois cohomology associated to the exact sequence
$$0\rightarrow R\rightarrow R[\Delta]\rightarrow
R[\Delta]/(R\cdot N)\rightarrow 0$$ gives  the exact sequence
$$0\rightarrow
R\xrightarrow{\cong} R\cdot N \rightarrow (R[\Delta]/(R\cdot N))^{\Delta}\rightarrow
\Hom(\Delta, R)\rightarrow 0.$$
Therefore $(R[\Delta]/(R\cdot
N))^{\Delta}\simeq \Hom(\Delta, R)$. We now describe this isomorphism more explicitly.

Let $\psi\in \Hom(\Delta, R)$, we define $$f(\psi):=\sum_{\sigma\in
\Delta}\psi(\sigma)\sigma\in R[\Delta]$$ and $\bar f(\psi)$ is the image of $f(\psi)$ in $R[\Delta]/R\cdot N$. Let $\sigma_0\in \Delta$, we have
$$\begin{aligned}\sigma_0(f(\psi))-f(\psi)
&=\sum_{\sigma}\psi(\sigma)\sigma_0\sigma-\sum_{\sigma}\psi(\sigma)\sigma \\
&=\sum_{\sigma}(\psi(\sigma)-\psi(\sigma_0))\sigma-\sum_{\sigma}\psi(\sigma)\sigma\\
&=-\psi(\sigma_0)\sum_{\sigma}\sigma\in R\cdot N,
\end{aligned}$$
hence $\bar f(\psi) \in (R[\Delta]/(R\cdot N))^{\Delta}$. On the other hand, we suppose $\bar f(\psi_1)=\bar f(\psi_2)$. Then
$\sum_{\sigma}(\psi_1(\sigma)-\psi_2(\sigma))\sigma\in R\cdot N$, $i.e.$, $\psi_1(\sigma)-\psi_2(\sigma)$ is a constant for all
$\sigma\in \Delta$. Let $e$ be the identity of $\Delta$, we have
$\psi_1(e)-\psi_2(e)=0$. Then $\psi_1=\psi_2$. Therefore we have
\begin{equation}\label{compute-R}
[R[\Delta]/(R\cdot N)]^{\Delta}=\{\bar f(\psi)\in R[\Delta]/R\cdot N: \psi\in \Hom(\Delta, R)\}.
\end{equation}

Since $K/K^g$ is cyclic, we can choose a primitive
character $\chi$ of $\Gal(K/K^g)$, then $\Res(\chi)$ is also a primitive
character of $\Gal(L.K/L.K^g)$, $i.e.$, a generator of $R$ (see (\ref{dif-R}) for the definition of $R$). Recalled
$b:=[K:K\cap (L.K^g)]=\# R$.
Rewrite the equation (\ref{compute-R}), we have
\begin{equation}\label{recompute-R}[R[\Delta]/(R\cdot N)]^{\Delta}=\{\bar f'(\psi)\in R[\Delta]/R\cdot N: \psi\in \Hom(\Delta, \Z/b\Z)\},\end{equation}
where $f'(\psi):=\sum_{\sigma\in
\Delta}\psi(\sigma)\Res(\chi)\cdot\sigma\in R[\Delta]$.

For any field $k$ with $char(k)=0$, using Grothendieck's purity theorem (cf.  \cite[Theorem 1.3.2]{CTSD94}), we can give a precise description of the map $\Br(X)\rightarrow H^1(k,\Pic(\overline X))$ by the observation 3 of \cite[Section 2]{Li69}.
Then we have the following commutative diagram $$\xymatrix{\Br_{vert}(X_{K^g}) \ar[dr] \ar[r]^{\phi} &R[\Delta]\ar[d] \\ &H^1(K_g,\Pic(\overline X)),}$$ where $\phi$ is induced by the residue map of $\Br(K^g(t))$. Using Faddeev's sequence (cf. \cite[Section 1.2]{CTSD94}), the equation (\ref{recompute-R}) implies that any element of $\Br(X_{K^g})$ $(=\Br_{vert}(X_{K^g}))$ which comes from $\Br(X)$ has the form $$\begin{aligned} &\rho+\sum_{\sigma\in
\Gal(L_g/k)}\psi(\sigma)\Cor_{L^\sigma/k}(t-\sigma(\eta),\Res(\chi))\\
&=\rho+\sum_{\sigma\in
\Gal(L_g/k)}\psi(\sigma)(p(t)^{\sigma},\chi),\end{aligned}$$
where $\rho\in \Br(K^g)$ and $\eta$ is a root of $p(t)$ in $\bar k^\times$. This argument  is similar as in the remark of  \cite[p. 76]{CHS1}.
\end{proof}

\begin{thm} \label{equal-X} Suppose
$char(k)=0$. Let $P(t)$ be an irreducible polynomial over $k$,
$K/k$ an abelian extension and $L=k[t]/(P(t))$. Let $X$ be the
CTHS partial compactification for $P(t)$ and $K$ as in Section 1, $X^c$ a smooth compactification of $X$. Then:
\begin{enumerate}[(a)]
\item The quotient $\Br(X)/\Br(X^c)$ is $2$-torsion.

\item $\Br(X^c)=\Br(X)$ if one of the following conditions holds:
\begin{enumerate}[(1)]
\item $\Gal(K/k)\simeq \Z/2^i\times A$, where $i\geq 0$ and $A$ has odd order;

\item $[L\cap K:k]$ is odd;

\item $[L:L\cap K]$ is even;

\item There exists $s\geq 1$ such that $2^{s}\mid [L:k]$ and the cokernel of the
multiplication by $2^{s-1}$ on $\Gal(K/k)$ has odd order;

\item $L/k$ contains an abelian subfield $L'/k$ with
$\Gal(L'/k)\cong(\Z/2\Z)^3$.
\end{enumerate}
\end{enumerate}

\end{thm}
\begin{proof} Let $k(X)$ be the function field of $X$. For each
discrete valuation ring $A$ which contains $k$ and with fraction field
$k(X)$ and residue field $\kappa_A$, there is a residue map
$$\partial_A: \Br(k(X))\rightarrow H^1(\kappa_A,\Q/\Z).$$ Since $char(k)=0$, Grothendieck's purity theorem (cf.  \cite[Theorem 1.3.2]{CTSD94}) gives
$$\Br(X^c)=\bigcap_{A}\Ker(\partial_A)\subset \Br(k(X)),$$
where $A$ runs through all discrete valuation ring as above.

Let $\mathcal{B}\in \Br(X)$, first we will prove $2\mathcal{B}\in
\Br(X^c)$, $i.e.$,  for any $g\in
\Gal(\bar\kappa_A/\kappa_A)$,
$$2\partial_A(\mathcal{B})(g)= 0\in \Q/\Z$$
by Grothendieck's purity theorem.

Since $k\subset
\kappa_A$, we can fix an embedding $\bar k\hookrightarrow \bar
\kappa_A$. Let $K^g$ be the fixed field of $K$ by $g$ and $L_g=L\cap K^g$. Let $f$ be the natural map
$\Br(X)\rightarrow \Br(X_{K^g})$. By Lemma \ref{form}, there exists
$\psi_\mathcal{B}\in \Hom(\Gal(L_g/k),\Z/b)$ such that
\begin{equation} \label{rep:B}f(\mathcal{B})=\rho+\sum_{\sigma\in
\Gal(L_g/k)}\psi_\mathcal{B}(\sigma)(p(t)^\sigma, \chi)
\end{equation}
where
$\rho\in \Br(K^g)$, $b=[K:K\cap (L.K^g)]$, $p(t)$ is a fix irreducible factor of $P(t)$ over $L_g$, and $\chi$ is a fixed primitive character of $\Gal(K/K^g)$.

Since $X$ is geometrically integral, the function field $K(X)
\cong k(X) \otimes_k K$ of $X\times_k K$ is finite and
unramified over $k(X)$. We can extend $A$ to a discrete valuation
ring $A_{K}$ of $K(X)$ with residue field
$\kappa_{A_{K}}=\kappa_A . K$. Indeed, the completion of
$k(X)$ for the given valuation is isomorphic to $\kappa_A((\pi))$,
where $\pi$ is a uniformizer. Considering the valuation given by $\pi$
on $(K . \kappa_A)((\pi))$ and using $K(X) = k(X) \otimes_k
K$, since $K(X)$ can naturally embed into $(K . \kappa_A)((\pi))$ and its image is dense in the latter, there is a discrete valuation ring $A_{K}\subset K(X)$ and its residue field is just $K . \kappa_A$.

For any intermediate field $M$ of $K/k$, we have similarly the
valuation ring $A_M$ of $M(X)$ with residue field $\kappa_{A_M}$. We
write $A_g$ for $A_{K^{g}}$.

By \cite[Proposition~1.1.1]{CTSD94}, we have the commutative
diagram
\begin{equation*}
    \begin{CD}
      \Br(k(X)) @>\partial_A>> H^1(\kappa_A,\Q/\Z)\\
      @V\Res_{k/K^{g}}VV @VV\Res_{\kappa_A/\kappa_{A_g}}V\\
      \Br(K^{g}(X)) @>\partial_{A_g}>> H^1(\kappa_{A_g},\Q/\Z).
    \end{CD}
\end{equation*}
Since $\kappa_{A_g} = \kappa_A . K^{g}$, we have $g \in
\Gal(\bar \kappa_A/\kappa_{A_g})$. Hence
\begin{equation*}
    \partial_A(\mathcal{B})(g)=\partial_{A_g}(f(\mathcal{B}))(g).
\end{equation*}
Therefore we only need to show $2 \partial_{A_g}(f(\mathcal{B})) = 0$.

We consider two cases:
\begin{enumerate}[(i)]
\item The case $ord_{A_g}(t)\geq 0$.\\
If $ord_{A_g}(p(t)^{\sigma})=0$ for all $\sigma \in \Gal(L_g/k)$, it is clear $\partial_{A_g}(f(\mathcal{B}))=0$ by (\ref{rep:B}).
If there is a $\sigma_0 \in \Gal(L_g/k)$
such that  $ord_{A_g}(p_1(t)^{\sigma_0})> 0$. It is clear that
$$ord_{A_g}(p_1(t)^\sigma)=0\text{ for }\sigma\neq \sigma_0$$ since
$p_1(t)^\sigma$ and $p_1(t)^{\sigma_0}$ are relatively prime. Then
we have $$ord_{A_g}(p_1(t)^{\sigma_0})=ord_{A_g}(P(t)).$$ Therefore
$$\begin{aligned}\partial_{A_g}(f(\mathcal{B}))=\psi_\mathcal{B}(\sigma_0)ord_{A_g}(p_1(t)^{\sigma_0})\cdot \bar \chi
=\psi_\mathcal{B}(\sigma_0)ord_{A_g}(P(t))\cdot \bar \chi =0
\end{aligned}$$ by the equation $N_{K/k}(\Xi)=P(t)$.

\item The case $ord_{A_g}(t)=\delta< 0$.\\
Denote $d:= deg(p(t))$. Then
$$ord_{A_g}(p(t)^{\sigma})=\delta\cdot deg(p(t)^{\sigma})=\delta d.$$
Therefore
$$\begin{aligned}\partial_{A_g}(f(\mathcal{B}))=\sum_{\sigma}\psi_\mathcal{B}(\sigma)ord_{A_g}(p(t)^{\sigma})\cdot \bar \chi
=\delta d\sum_{\sigma}\psi_\mathcal{B}(\sigma)\cdot \bar
\chi.\end{aligned}$$  Denote $\psi_\mathcal{B}(\Gal(L_g/k)):=\left<u\right>\subset
\Z/b$ with $u\mid b$. Then we have
$$\begin{aligned}\partial_{A_g}(f(\mathcal{B}))&=\delta d\cdot \# \Ker(\psi_\mathcal{B})\cdot u(1+2+\cdots +(b/u-1))\cdot \bar \chi\\
&=\delta d\cdot \# \Ker(\psi_\mathcal{B})\cdot b(b/u-1)/2\cdot \bar
\chi.\end{aligned}$$
Denote
$$\begin{aligned}m&:=[K:K^g]=[K:K\cap (L.K^g)]\cdot [K\cap (L.K^g):K^g]\\
&=b\cdot [K\cap (L.K^g):K^g].\end{aligned}$$ Note that $[L.K^g:K^g]=[L:L_g]=d$ , we denote \begin{equation} \label{def-d'}
\begin{aligned}d':=&d/[K\cap
(L.K^g):K^g]=[L.K^g:K\cap (L.K^g)]\\=&[L.K:K]=[L:L\cap
K].\end{aligned}
\end{equation} So we have
\begin{equation}\label{for:res}
\partial_{A_g}(f(\mathcal{B}))=\delta  d'\cdot \# \Ker(\psi_\mathcal{B})\cdot m(b/u-1)/2\cdot \bar
\chi.
\end{equation}
Since $m\cdot \chi=0$, we have
$$2\partial_{A_g}(f(\mathcal{B}))=0,$$
this concludes the proof that $2\mathcal{B} \in \Br(X^c)$.
\end{enumerate}

Our work shows that, if $2\mid \delta d'\cdot \#
\Ker(\psi_\mathcal{B})\cdot (b/u-1)$ in the formula (\ref{for:res}), then
$\partial_A(\mathcal{B})=0$ for any $A$, hence
$\mathcal{B}\in \Br(X^c)$. We will prove the part (b) case by case:
\begin{enumerate}[i)]
\item If $i=0$, then  the order $\Gal(K/k)$ is odd, hence $b/u$ ($=[L_g:k]$) is odd. Therefore
$2\mid (b/u-1)$.

Suppose $i>0$, write $\Gal(K/k)=H_1\times H_2$, where $H_1$ is cyclic
and of order $2^i$ and $H_2$ has odd order. By the K\"{u}nneth
formula (\cite[p. 96]{NSW}), we have $$H^2(K/k,
\Q/\Z)=\bigoplus_{i+j=2}H^i(H_1,H^j(H_2,\Q/\Z)).$$ Since the orders
of $H_1$ and $H_2$ are relatively prime, we have
$$H^1(H_1,H^1(H_2,\Q/\Z))=0.$$ Obviously $H^2(H_1,\Q/\Z)=0$ since $H_1$ is
cyclic. Then
$$H^2(K/k, \Q/\Z)=H^2(H_2, \Q/\Z).$$ We have
$$H^1(k,\Pic(\overline X))\simeq\cyr{X}^2_\omega (\widehat T )_P$$ by the sequence (\ref{bs}) and Proposition \ref{abel}. Denote $h=\# H_2$, then $$h\cdot H^1(k,\Pic(\overline X))=0$$ by the fact $\cyr{X}^2_\omega (\widehat T )\simeq H^3(K/k, \Z)\simeq H^2(K/k, \Q/\Z)$.

 Let $\mathcal{B} \in \Br(X)$. Since
$\Br(X)/\Br_0(X)\hookrightarrow H^1(k,\Pic(\overline X))$ and $\partial_A(\Br_0(X))=0$, one
has
$$h
\partial_A(\mathcal{B})=0\in H^1(\bar \kappa_A/\kappa_A,\Q/\Z).$$
On the other hand we have  $$2
\partial_A(\mathcal{B})=0\in H^1(\bar \kappa_A/\kappa_A,\Q/\Z)$$ by the proof of part a). Since $(2,h)=1$, we have  $
\partial_A(\mathcal{B})=0$, hence $\mathcal{B}\in
\Br(X^c)$.

\item Since $L\cap K/k$ has odd degree, then $L_g/k$ also has odd degree, $i.e.$, $b/u$ is
odd, then $2\mid (b/u-1)$.

\item Obviously $2\mid d'$ by (\ref{def-d'}).

\item If $2^s\nmid[L\cap K:k]$, then $[L:L\cap K]$ is even, hence $2\mid d'$. So we may assume $2^s\mid
[L\cap K:k]$.

Denote by $v_2(n)$ the $2$-adic valuation for any $n\in \Z$, and
denote $i:=v_2([L\cap K:L_g])$. Obviously
\begin{equation} \label{v2:Lg} v_2([L_g:k])=v_2([L\cap K:k])-v_2([L\cap K:L_g])\geq s-i.\end{equation}
Since $K/K^g$ is cyclic and the assumption for
 $\Gal(K/k)$, we have $v_2([K:K^g])\leq s-1$. Recall $b=[K:K\cap
(L.K^g)]$, and note that $$(L\cap
K).K^g\subset K\cap (L.K^g)\text{ and }[(L\cap K).K^g:K^g]=[L\cap K:L_g],$$ we have
\begin{equation} \label{v2:b} \begin{aligned} v_2(b)&\leq v_2 ([K:(L\cap K).K^g])=v_2 ([K:K_g])-v_2 ([(L\cap K).K^g:K^g])\\
&\leq s-i-1.
\end{aligned}
\end{equation}
By (\ref{v2:Lg}) and (\ref{v2:b}), the kernel of
$\psi_{\mathcal{B}}: \Gal(L_g/k)\rightarrow \Z/b\Z$ has order divisible by $2$, $i.e.$, $2\mid \#
\Ker(\psi_\mathcal B)$.

\item If $L'\not \subset L\cap K$, then $2\mid[L:L\cap K]$, $i.e.$,
$2\mid d'$. So we may assume $L'\subset L\cap
K$. Since $K/K^g$ is cyclic, $L\cap K/L_g$ is also cyclic. Then there is
a subfield $L''$ of $L_g\cap L'$ with
$\Gal(L''/k)\cong\Z/2\Z\times \Z/2\Z$, hence $2\mid
\#\Ker(\psi_\mathcal{B})$. \qedhere
\end{enumerate}
\end{proof}

Let $X$ be the variety in the following proposition, then $\Br(X)/\Br_0(X)=\Z/n$. However, if $n$ is even,  $\Br(X^c)/\Br_0(X^c)=\Z/(n/2)$ by the following proposition, hence $\Br(X^c) \neq \Br(X)$ in this case.
\begin{prop} \label{Q_1} Suppose $char(k)=0$ and $H^3(k,\bar k^\times)=0$.
Let $n$ be an integer. Let
$K/k$ be an abelian extension with $\Gal(K/k)=\Z/n\times \Z/n$.
Suppose $L\subset K$ and $L/k$ is cyclic of order $n$. Let $P(t)$ be an irreducible polynomial over k
such that $L\cong k[t]/(P(t))$. Let $X$ be the CTHS partial compactification for $P(t)$ and $K$ (see Section 1 for definition). Then
$$\Br(X^c)/\Br_0(X^c)=\begin{cases} \Z/(n/2) \ \ &\text{if } n  \text{ is even},\\
\Z/n\ \ &\text{if } n \text{ is odd}.
\end{cases}$$
\end{prop}
\begin{proof} Since
$K/L$ is cyclic, we have
$$\cyr{X}^2_\omega(\widehat T)_P=\cyr{X}^2_\omega(\widehat T)\simeq H^3(\Z/n\times \Z/n,
\Z)\simeq \Z/n,$$
the last equation follows from  the K\"{u}nneth
formula (\cite[p. 96 ]{NSW}).
By the sequence (\ref{bs}) and Proposition \ref{abel}, we have
$$H^1(k,\Pic(\overline X))\simeq \cyr{X}^2_\omega(\widehat T)_P\simeq\Z/n.$$
Since $H^3(k,\bar k^\times)=0$, we have $$\Br(X)/\Br_0(X)\simeq
H^1(k,\Pic(\overline X))\simeq  \Z/n.$$

If $n$ is odd, the result immediately follows from the case (a) in
Theorem \ref{equal-X}. Then we only need to consider the case $n$ is
even.

Write $n=2^sn_0$ where $n_0$ is odd and $s \geq 1$.
 By Theorem \ref{equal-X}, the
cokernel of  $\Br(X^c)/\Br_0(X^c)$ in $\Br(X)/\Br_0(X)$ is killed by
2. Then we only need to show that an element $B$ in
$\Br(X)/\Br_0(X)$ with order $2^s$ is not
contained in $\Br(X^c)/\Br_0(X^c)$.

Let $\left<g\right>$ be the cyclic subgroup of $\Gal(K/k)$ generated by $g$. Let
$K^g$ be the fixed field of $\left<g\right>$ in $K$. We have the
following morphism
$$f_g: H^1(K/k,\Pic(X_K))\rightarrow H^1(K/K^g,\Pic(X_K)).$$
Let $T'$ be the $k$-torus defined by $N_{L/k}(\Xi_1)\cdot
N_{K/k}(\Xi_2)=1$. Let $f:=\prod_{\left<g\right>}f_g$. Since
$$\Ker(f)= \cyr{X}^1_\omega (K/k,\Pic(X_K))=\cyr{X}^1_\omega (\Pic(\overline X))=\cyr{X}^2_\omega (\widehat T')=0$$
by Proposition \ref{compact-omega} for the last second equality and Corollary \ref{Sha-T} for the last equality, $f(B)$ has order $2^s$.

%
%

There exists $\left<g_0\right>$ such that $f_{g_0}(B)$ has order divisible by $2^s$ since $f(B)$ has order $2^{s}$. Denote $L_{g_0}:=L\cap K^{g_0}$. Then both $K/L.K^{g_0}$  and $L_{g_0}/k$ have order divisible by $2^s$ by Lemma \ref{form}, since $f_{g_0}(B)$ has order divisible by $2^s$. Let $\Theta_0/k$ be the unique subfield of $K$
with $\Gal(\Theta_0/k)=\Z/n_0\times\Z/n_0$ and let
$\Theta=K^{g_0}.\Theta_0$.  Then we can see $\Theta=L .\Theta_0$ and $K/\Theta$ has degree $2^s$.

By the inflation-restriction sequence, the kernel of the map
$$H^1(K/K^{g_0},\Pic(X_K))\rightarrow H^1(K/\Theta,\Pic(X_K))$$
has odd order since $\Theta/K^{g_0}$ has odd degree. Since $f_{g_0}(B)\in H^1(K/K^{g_0},\Pic(X_K))$ has
order divisible by $2^s$ and $K/\Theta$ has degree $2^s$, $f_\Theta(B)\in H^1(K/\Theta,\Pic(X_K))$ has order $2^s$ where $f_\Theta:H^1(K/k,\Pic(X_K))\rightarrow H^1(K/\Theta,\Pic(X_K))$.

Let $\mathcal B \in \Br(X)$ which lifts $B$. Let
$f_\Theta': \Br(X)\rightarrow \Br(X_{\Theta})$ be the natural map. By Lemma \ref{form}, we have
$$f_{\Theta}'(\mathcal B)=\rho+\sum_{\sigma\in
\Gal(L/k)}\psi_{\mathcal B} (\sigma)(t-\sigma(\eta), \chi)$$ where
$\rho \in \Br(\Theta)$, $\eta$ is a root of $P(t)$ in $L$,
$\psi_{\mathcal B}\in \Hom(\Gal(L/k),\Z/2^s)$, and $\chi\in
\Hom(\Gal(K/\Theta),\Z/2^s)$ is primitive. Since $f_\Theta(B)$ has order $2^s$,
we deduce $\psi_{\mathcal B}$ has order divisible by $2^s$. Since $\Gal(L/k)$ is cyclic of order
$n=2^s\cdot n_0$ with $n_0$ odd, then $\psi_\mathcal B$ has order $2^s$.

The variety $X_{\Theta}$ contains an open affine $\Theta$-subvariety
$U$ defined by
$$\prod_{i}N_{K/\Theta}(\Xi_i)=P(t) \text{ and }t\neq 0,$$
where $K\otimes_k \Theta\cong \prod_{i} K$ and $K/\Theta$ has degree $2^s$. Let $Q(u)=u^nP(1/u)$. Let $W$ be the
smooth affine $\Theta$-variety defined by
$$\prod_{i}N_{K/\Theta}(\Xi'_i)=Q(u).$$
Obviously the open subvariety of $W$ defined by $u\neq 0$ is
isomorphic to $U$ by the map $u\mapsto 1/t, \Xi'_1 \mapsto
\Xi_1/t^\mu, \Xi'_i\mapsto \Xi_i \text{ for }i\geq 2$, where
$\mu=n/2^s$.

Let $D$ be the divisor of $W$ over $\Theta$ defined by $u=0$.
It is easy to see that the divisor $D$ is geometrically irreducible. So
we have $\bar \Theta \cap \kappa_D=\Theta$, where $\kappa_D$ is the function field of $D$. The local ring $A_D$
associated with $D$ is a discrete valuation ring with
$ord_{A_D}(t)=-1$ and $\kappa_D=\kappa_{A_D}$. So $ord_{A_D}(t-\sigma(\eta))=-1$ for all
$\sigma$. Therefore
$$\partial_{A_D}(f_\Theta '(\mathcal B))=-\sum_{\sigma\in \Gal(L/k)}\psi_{\mathcal
B}(\sigma)\cdot \bar \chi.$$ Since $\psi_\mathcal B$ has order $2^s$, then $\psi_{\mathcal B}:\Gal(L/k)\rightarrow \Z/2^s$ is surjective, hence
$$\sum_{\sigma\in \Gal(L/k)}\psi_{\mathcal B}(\sigma)=n_0\cdot
2^{s-1}(2^{s}-1)\equiv -n_0\cdot 2^{s-1} \in \Z/2^s.$$ Since $\chi$
has order $2^s$, we have
$$\sum_{\sigma\in \Gal(L/k)}\psi_{\mathcal B}(\sigma)\cdot \chi=-n_0\cdot 2^{s-1}\cdot \chi \neq
0\in \Hom(\Gal(K/\Theta),\Q/\Z).$$ Since $\bar \Theta\cap
\kappa_{A_D}=\Theta$, we have
$$\partial_{A_D}(f_\Theta'(\mathcal B)) \neq 0\in H^1(\kappa_{A_D},\Q/\Z),$$
Then $f_\Theta '(\mathcal B)\notin \Br(X^c_\Theta)$, hence $\mathcal B\notin \Br(X^c)$.
\end{proof}

\begin{rem*}
In the case $n=2$, we get $\Br(X^c)=\Br_0(X^c)$, this answers
the
final
question
in \cite[Questions on p. 82 and p. 83]{CHS1}.
\end{rem*}

\bigskip

\subsection{A case with $K/k$ a Galois extension and $P(t)$ with multiple rational roots}

The part of this section is to prove Proposition \ref{brauer-split}, which is motivated by the
Question  (a) in \cite[p. 82]{CHS1}.

Let $G$ be a finite group and $M$ a $G$-module. Define
$$\cyr{X}^2_\omega(G,
M):=\bigcap_{g\in G}\Ker[H^2(G,M)\rightarrow H^2(\left<g
\right>,M)].$$
\begin{lem}\label{p-finite} Let $p$ be a prime, $G=\Z/p\times \Z/p$ and $\Z/p$ the trivial $G$-module. Then $$\cyr{X}^2_\omega(G,
\Z/p)=\begin{cases}0 \ \ &\text{if } p=2,\\
\Z/p\ \ &\text{if } p \text{ is odd},\end{cases}$$
\end{lem}
\begin{proof} Denote $R:=\Z/p$. Let $H_1=\left<g_1\right>,H_2=\left<g_2\right>$ and $G=H_1\times
H_2$. Let $N=N_1+N_2$ with $N_i=\sum_{\sigma\in G/H_i}\sigma$ for
$i=1,2$. Denote
$$\Delta:=(R[G/H_1]\oplus R[G/H_2])/(R\cdot N).$$ Then we have the
following exact sequence
$$0\rightarrow R\rightarrow R[G/H_1]\oplus R[G/H_2]\rightarrow \Delta\rightarrow 0.$$
Since the restriction map $$H^1(G,\Z/p)\rightarrow
H^1(H_1,\Z/p)\times H^1(H_2,\Z/p)\text{ is surjective},$$ we have the
following exact sequence
$$0\rightarrow H^1(G,\Delta)\rightarrow H^2(G,\Z/p)\rightarrow
H^2(H_1,\Z/p)\times H^2(H_2,\Z/p).$$ Let $H$ be a non-trivial
cyclic subgroup of $G$ and $H\neq H_1,H_2$. It is easy to verify
that $R[G/H_i]\simeq R[H]$ as $H$-modules for $i=1,2$. Then
$$H^i(H, R[G/H_1]\times R[G/H_2])=0 \text{ if } i\geq 1.$$
Therefore we have the following commutative diagram
$$
\xymatrix @R=12 pt @C=10pt{
0 \ar[r] &
H^1(G,\Delta) \ar[r] \ar[d] &
H^2(G,\Z/p)   \ar[r] \ar[d] & H^2(H_1,\Z/p)\oplus H^2(H_2,\Z/p) \ar[d] \\
0 \ar[r] &  H^1(H,\Delta) \ar[r] &  H^2(H,\Z/p)    \ar[r] & 0 .
}
$$
Therefore we have $$\cyr{X}^2_\omega(G, \Z/p)\cong \bigcap_{H\neq
H_1,H_2}\Ker[H^1(G,\Delta)\rightarrow H^1(H,\Delta)].$$ By the
inflation-restriction sequence, we have
$$H^1(G/H,\Delta^{H})=\Ker[H^1(G,\Delta)\rightarrow H^1(H,\Delta)].$$
Since $G/H$ is cyclic and $\Delta$ is finite, using Tate cohomology one has $$\mid
H^1(G/H,\Delta^{H})\mid=\mid\hat
H^0(G/H,\Delta^{H})\mid=\mid\Delta^{G}/N_{G/H}(\Delta^{H})\mid.$$
It is easy to verify that $\Delta^{G}\simeq \Z/p$ by the exact
sequence $$0\rightarrow \Z/p\rightarrow R[G/H_1]\oplus
R[G/H_2]\rightarrow \Delta\rightarrow 0.$$ Furthermore, we have
$$0\rightarrow \Z/p\rightarrow \Z/p\cdot N_1\oplus \Z/p\cdot N_2\rightarrow \Delta^{H} \rightarrow H^1(H,\Z/p)\rightarrow 0.$$
By a similar argument as for the equation (\ref{compute-R}), $\Delta^{H}$ is generated by $$R\cdot N_1\oplus R\cdot
N_2\text{ and }\{\sum_{g\in H}\chi(g)(\bar g,\bar g): \chi \in
H^1(H,\Z/p)\},$$ where $(\bar g,\bar g)\in G/H_1\times G/H_2$.
Since $G/H$ trivially acts on $R\cdot N_1\oplus R\cdot N_2$, we
have
$$N_{G/H}(R\cdot N_1\oplus R\cdot N_2)=p\cdot(R\cdot N_1\oplus R\cdot
N_2)=0.$$ Note that $H_1\simeq G/H$ since $G=\Z/p\times \Z/p$ and $p$ is
a prime. Denote $$u=\sum_{g\in H}\chi(g)(\bar g,\bar g)\in \Delta,
\text{ where } \chi \text{ is a primitive character}.$$ Then
$$\sigma(u)= \sum_{g\in H}\chi(g)(\bar g,\overline {\sigma
g})\text{ where } \sigma\in H_1.$$ Hence we have
$$\begin{aligned}N_{G/H}(u)&=\sum_{\sigma\in H_1}\sum_{g\in H}\chi(g)(\bar g,\overline {\sigma
g})=\sum_{g\in H}\chi(g)\sum_{\sigma\in H_1}(\bar g,\overline
{\sigma g})\\
&=\sum_{g\in H}\chi(g)(0,\sum_{\sigma\in H_1}\overline {\sigma
g})=(0,\sum_{g\in H}\chi(g) N_2)\\
&=\begin{cases}(0,N_2) \ \ &\text{if } p=2,\\
0\ \ &\text{if } p \text{ is odd}.\end{cases}
\end{aligned}$$
Therefore we have $$\mid
H^1(G/H,\Delta^{H})\mid=\mid\Delta^{G}/N_{G/H}(\Delta^{H})\mid=\begin{cases}0 \ \ &\text{if } p=2,\\
p\ \ &\text{if } p \text{ is odd}.\end{cases}$$ If $p=2$, then we have
$\cyr{X}^2_\omega(G, \Z/p)=0$.

By the K\"{u}nneth formula (\cite[p. 96 ]{NSW}), we have $$H^2(G,\Z/p)\simeq
\bigoplus_{i+j=2} H^i(H_1,H^j(H_2,\Z/p)).
$$ Therefore $$H^1(G,\Delta)\simeq \Ker[H^2(G,\Z/p)\rightarrow H^2(H_1,\Z/p)\oplus H^2(H_2,\Z/p)]\simeq \Z/p.$$
Suppose $p$ is odd. By the above argument,
$$\Ker[H^1(G,\Delta)\rightarrow H^1(H,\Delta)]\simeq
H^1(G/H,\Delta^{H})\simeq \Z/p.$$ Therefore the restriction map
$H^1(G,\Delta)\rightarrow H^1(H,\Delta)$  is a zero map, hence $\cyr{X}^2_\omega(G, \Z/p)\simeq H^1(G,\Delta)\simeq \Z/p$.
\end{proof}

\begin{prop} \label{brauer-split} Suppose $H^3(k,\bar k^\times)=0$. Let $P(t)=c\prod_{i=1}^m(t-e_i)^{d_i}$ with $c\in k^\times$, $e_i\in k$ distinct and  $d:=\text{gcd}(d_1,\cdots,d_m)$.
Let $K/k$ be a Galois extension with $\Gal(K/k)$  killed by $d$.
Let $X$ be the CTHS partial compactification for $P(t)$ and $K$
(see Section 1 for definition), and $X^c$ a smooth compactification of $X$. Then
$$\Br(X^c)/\Br_0(X^c)\cong \cyr{X}^2_\omega(K/k, \Z/d).$$
\end{prop}
\begin{proof} Let $\left<g\right>\subset \Gal(K/k)$
be a cyclic subgroup. Let $K^g$ be the fixed field of
$\left<g\right>$ in $K$. We have the following morphism
$$
f_g: \Br(X)/\Br_0(X) \to \Br(X_{K^g})/\Br_0(X_{K^g}).
$$
Since $[K:K^g]\mid d$ by our assumption, the $K^g$-variety
$X_{K^g}$ is $K^g$-birationally isomorphic to $\A^1_{K^g}\times Y$, where $Y$ is
a $K^g$-variety defined by the equation $c=\prod_{i}N_{K/K^g}(\Xi_i)$ and $K\otimes_k K^g\cong \prod_{i} K$. We know
$H^1(K^g,\Pic(\overline{Y^c}))=0$ since $K/K^g$ is cyclic, which implies $\Br(X^c_{K^g})/\Br_0(X^c_{K^g})=0$.

If $B\in
\Br(X^c)/\Br_0(X^c)$, we have $f_g(B)\in
\Br(X^c_{K^g})/\Br_0(X_{K^g})$, hence $f_g(B)=0\in
\Br(X_{K^g})/\Br_0(X_{K^g}^c)$. Let $f: =\prod_{\left<g\right>}f_g$, so
$$\Br(X^c)/\Br_0(X^c)\subset \Ker(f).$$
Let $T'$ be the torus over $k$ defined by $t_1^{d_1}\cdots
t_m^{d_m}N_{K/k}(\Xi)=1$.     By Proposition \ref{compact-omega}, we
know $\Ker(f)\simeq \cyr{X}^2_\omega (K/k,\widehat T' )$ is
contained in $\Br(X^c)/\Br_0(X^c)$ since $H^3(k,\bar k^\times)=0$. So $\Br(X^c)/\Br_0(X^c)=\Ker(f)$ and we only need to show
$$\cyr{X}^2_\omega (K/k,\widehat T' )\simeq \cyr{X}^2_\omega
(K/k,\Z/d).$$

Let $N=\sum_{\sigma\in \Gal(K/k)}\sigma$. We know $$\widehat
T'\simeq[(\Z\oplus \cdots \oplus \Z)\oplus \Z[K/k]]/\Z\cdot ((d_1,\cdots,
d_m)+ N).$$ Denote
$$M:=(\Z\oplus \cdots \oplus \Z)/\Z\cdot (d_1,\cdots, d_m).$$
This exact sequence $$0\rightarrow
\Z[K/k]\rightarrow \widehat T'\rightarrow M\rightarrow 0$$ implies $\cyr{X}^2_\omega (K/k,\widehat T' )\simeq \cyr{X}^2_\omega
(K/k,M)$. Since $M\cong\Z/d\times M'$ where $M'$ is a free
$\Z$-module, we have $\cyr{X}^2_\omega (K/k,M)\simeq
\cyr{X}^2_\omega (K/k,\Z/d)$ by the fact $\cyr{X}^2_\omega (K/k,M')=0$.
\end{proof}

By Lemma \ref{p-finite} and Proposition \ref{brauer-split}, we have
the following result:

\begin{corollary}\label{Q_2} Suppose $H^3(k,\bar k^\times)=0$. Let $p$ be a prime and $P(t)=c\prod_{i=1}^m(t-e_i)^p$ with $c\in k^\times$, $e_i\in k$ distinct. Let $K/k$ be an abelian
extension with $\Gal(K/k)=\Z/p\times \Z/p$. Then $$\Br(X^c)/\Br_0(X^c)=\begin{cases}0 \ \ &\text{if } p=2,\\
\Z/p\ \ &\text{if } p \text{ is odd}.\end{cases}$$
\end{corollary}
\begin{rem*}
For $p=2$, this corollary answers   Question  (a) in \cite[p.
82]{CHS1}.
\end{rem*}

\section{Rational points under Schinzel's hypothesis}\label{Schinzelsection}

Let $k$ be a number field. Let $\Omega_k$ be the set of all places of $k$. 
Let $P(t)$ be a nonzero polynomial over $k$.
In this section we mainly consider the question whether the
Brauer-Manin obstruction is the only obstruction to the Hasse
principle and weak approximation for a smooth and proper model of
the $k$-variety defined by
\begin{equation*}N_{K/k}(\Xi)=P(t)
\end{equation*}
where $\Xi$ is a "variable" in $K$, $K/k$ is a finite field extension
and $N_{K/k}$ is a formal norm associated with the extension for the
variable $\Xi$.

A positive answer was given by Colliot-Th\'el\`ene, Sansuc and
Swinnerton-Dyer in their remarkable paper (\cite{CTSS87}) when $P(t)$ has degree $\leq 4$ and $K/k$ is a quadratic extension. A positive
answer is also known for the following cases: the extension $K/k$ has degree 3 and the
polynomial $P(t)$ has degree at most 3 (\cite{CTSal89}); the polynomial $P(t)$ having just two roots in $k$ and $K/k$ arbitrary (\cite{CHS1,HBSK,SJ}); $k=\Q$, $P(t)$ has degree $2$ and $K/\Q$ arbitrary  (\cite{BH,DSW}); $k=\Q$, $P(t)$ is split in $\Q$ and $K/\Q$ arbitrary (\cite{BM,BMS,HSW}).

Conditional results have been obtained under Schinzel's hypothesis
(H). Under this hypothesis, Colliot-Th\'el\`ene, Skorobogatov and
Swinnerton-Dyer (\cite[Thm. 1.1 (e)]{CTSSD98}) proved that the
Brauer-Manin obstruction is the only obstruction to the Hasse
principle and weak approximation for smooth projective models of
varieties defined by an equation $N_{K/k}(\Xi)=P(t)$ when the
extension $K/k$ is abelian and norm equations $N_{K/k}(\Xi)=c$ (for
any $c\in k^\times$) satisfy the Hasse principle and weak approximation,
for instance when $K/k$ is cyclic. 
Their result is more general, they consider smooth projective varieties over $k$ which admit a fibration over the projective line
such that

(a)  A certain  abelianity condition on the splitting field of the singular fibres
holds (Condition (i) in \cite[Thm. 1.1]{CTSSD98}).

(b) The Hasse principle and weak approximation hold on the smooth
fibres.

 In this section we shall handle three new classes  of fibrations over the
 projective line whose generic fibre is birationally a principal
 homogeneous space under a torus. In each of these classes
one of conditions (a) or (b)  is not in general fulfilled.

The following observation is important for the proof of Theorem \ref{rational-1} and \ref{cycle-1}.
\begin{lem}\label{Brauer-trivial}
Let $k$ be a number field and $K/k$ a finite field extension. Let $P(t)$
be a nonzero polynomial over $k$. Let $X$ be the CTHS partial compactification (see Section 1 for definition) of the equation $N_{K/k}(\Xi)=P(t)$  and $X_K:=X\times_k K$. Then $\Br(X_K)=\Br_0(X_K)$.
\end{lem}
\begin{proof} Denote by $T$ the torus $R^1_{K/k}(\G_m)$.  Let $T_K:= T\times_k K$, then $\widehat{T_K}$ is a permutation $\Gamma_K$-module, hence $\cyr{X}^2_\omega (K,\widehat{T_K})=0$ and $H^1(K,\Z_P \otimes \widehat{T_K})=0$. Therefore $H^1(K,\Pic(\overline{X}))=0$ by the sequence (\ref{bs}), hence $\Br(X_K)=\Br_0(X_K)$.
\end{proof}

\begin{thm}\label{rational-1}
Let $k$ be a number field and $K/k$ an abelian extension. Let $P(t)$
be a nonzero polynomial over $k$ and $T=R^1_{K/k}(\G_m)$. Suppose
$\cyr{X}_\omega^2(\widehat T)_P=\cyr{X}_\omega^2(\widehat T)$ (see
Section 1 for definition). Assume Schinzel's hypothesis holds. Then the
Brauer-Manin obstruction to the Hasse principle and weak
approximation for rational points is the only obstruction for any smooth proper model of
the variety over $k$ defined by the equation
\begin{equation} \label{equation:1}
N_{K/k}(\Xi)=P(t).
\end{equation}
\end{thm}
\begin{proof} It is sufficient to prove the statement for any given model. Let $V$ be the smooth locus of the affine $k$-variety
defined by (\ref{equation:1}).  Let $Y$ be a smooth
compactification of $V$ with a projection $p: Y\rightarrow \P^1_k$
defined by $(\Xi,t)\mapsto t$.

We assume that $Y$
has points in all completions of $k$, and we are given a finite set
$S$ of places of $k$ which contains all archimedean places, and points $P_v\in Y(k_v)$ for $v\in S$. We
assume that there is no Brauer-Manin obstruction to weak
approximation for $(P_v)_{v\in S}$. This means that
we may complete the family $(P_v)_{v\in S}$ to a family $(P_v)_{v\in
\Omega_k}$ such that
\begin{equation}\label{Br}\text{for any }\mathcal {B}\in \Br(Y), \sum_{v\in
\Omega_k}inv_v(\mathcal{B}(P_v))=0\in \Q/\Z.\end{equation} From this,
we want to deduce that there exists $P\in Y(k)$ as close as we wish
to each $P_v\in Y(k_v)$ for $v\in S$.

Let $U_0\subset \A^1_k$ be a non-empty Zariski open set defined by
$P(t)\neq 0$.  By the continuity of the pairing
of $Y(k_v)$ with elements of the Brauer group, we can replace each
$P_v$ by a close enough point such that condition (\ref{Br}), which
only involves finitely many classes in the Brauer group, will not be
affected. Therefore we can assume that all points appearing in (\ref{Br})
lie in $p^{-1}(U_0)$ by the implicit function theorem.

Let $Q_v:=p(P_v)\in U_0(k_v)\subset \A^{1}(k_v)$. We first choose a
$k$-point $Q_0\in U_0(k)$, close enough to each $Q_v$ for each
archimedean place $v$, and such that $Q_0$ is different from each
$Q_v$ for $v\in S$. Let $u=1/(t-Q_0)$. Denote
$Q(u):=u^{deg(P(t))}P(1/u+Q_0)$ and $n:=[K:k]$. Choose $l$ sufficient large
such that $ln-deg(P(t))>0$ and denote $P^{(1)}(t):=t^{ln-deg(P(t))}Q(t)$. Then we get a new smooth
affine variety  $W$ defined by
$$N_{K/k}(\Xi)= P^{(1)}(t)\neq 0,$$ and it is
isomorphic to the open subvariety of $p^{-1}(U_0)$ defined by $t\neq
Q_0$ which contains all $P_v$ for $v\in S$. Let $W_0\subset \A^1_k$ be a
non-empty Zariski open subset defined by $P^{(1)}(t)\neq 0$. Then $p$
projects $W$ to $W_0$.

We are now looking for a point $Q\in W_0(k)$ with associated
coordinate $\lambda\in k$, such that $\lambda$ is very close to each
$Q_v$ for $v\in S$, $v$ finite, $\lambda$ is big enough at each
archimedean place of $k$ ($Q$ is close enough to the infinite point), and such that the fibre $W_Q$ of $p$ has a $k$-rational point.

Let $X$ be the \emph{CTHS} partial compactification (see Section 1
for definition) of $V$.
We know $H^1(k, \Pic(\overline X))\rightarrow \cyr{X}^2_\omega (\widehat T
)$ is surjective by our assumption and the sequence (\ref{bs}). Then we can choose
a finite subset $B\subset \Br(X)$ such that the image of $B$ by the
composite map $$\Br(X) \rightarrow H^1(k, \Pic(\overline X))\rightarrow
\cyr{X}^2_\omega (\widehat T)$$ is $\cyr{X}^2_\omega
(\widehat T)$.

Choose $\omega_1,\cdots,\omega_n\in \frak o_K$ to be a basis of $K$ over $k$, where $\frak o_K$ (resp. $\frak o_k$) is the ring of integers of $K$ (resp. $k$).
Choose $S'$ to be a finite set of places of $k$ containing all archimedean places such that $P^{(1)}(t) \in \frak o_{S'}[t]$ and $\omega_1,\cdots,\omega_n$ generate $\frak o_K\otimes_{\frak o_k} {\frak o_{k_v}}$  over $\frak o_{k_v}$ for any $v \notin S'$. Let $\mathcal W \subset \A_{\frak o_{S'}}^{n+2}$ be the smooth affine integral model of $W$ over $\frak o_{S'}$ defined by
$$N_{K/k}(x_1 \omega_1+\cdots+x_n\omega_n)=P^{(1)}(t) \text{ and } y\cdot P^{(1)}(t)=1.$$
We enlarge $S'$ such that for any $\mathcal B\in B$, $\mathcal B$ can be extended to $\Br(\mathcal W)$.

Now we enlarge $S$ such that $S\supset S'$, and that $S$ contains all the original places at which
we want to approximate, and that $S$ also contains the places
associated in Hypothesis $H_1$ (\cite[p. 71]{CTSD94}) to the
polynomials of $P^{(1)}(t)$, all ramified places of $K/k$ and all
places $v$ such that for some $\mathcal B \in B$, $inv_w(j_K(\mathcal B))\neq 0$ where
$w$ is a place of $K$ over $v$ and $j_K: \Br(X) \to \Br(X_K)$, noting that $\Br(X_K)=\Br_0(X_K)$ by Lemma \ref{Brauer-trivial}.

Let $\{p_i(t)\mid 1\leq i\leq m\}$ be all irreducible terms of
$P^{(1)}(t)$ over $k$ and let $$A=\{(p_i(t),\chi)\in \Br(W): 1\leq i\leq
m, \chi
 \in \Hom(\Gal(K/k),\Q/\Z)
 \}.$$
Since condition (\ref{Br}) holds,
there is no Brauer-Manin obstruction on $Y$ to weak approximation for
$(P_v)_{v\in S}$. According to Harari's formal lemma (see
\cite{Ha94}), we may find a finite set $S_1$ of places of $k$,
containing $S$, and points $P_v\in W(k_v)$, $v\in S_1$, and which
extend the given family
$$P_v \in W(k_v),v\in S,$$ such that for each $\mathcal B\in
A\cup B$
\begin{equation}\label{Ha:3.1}
\sum_{v\in S_1}inv_v(\mathcal B(P_v))=0.
\end{equation}

Now apply Hypothesis $(H_1)$ (\cite[Proposition
4.1]{CTSD94}), we thus find $\lambda\in k$ close enough to each
$Q_v$ for all finite places $v\in S_1$, $\lambda$ integral away from
$S_1$, and $\lambda$ as large as one wants at all archimedean places, so
that:

(i) The fibre $W_\lambda$ of $p$ contains a $k_v$-point $P_v'$ which
is as close as we wish to $P_v$ for all places $v\in S_1$, and such
that
$$inv_v(\mathcal B(P_v'))=inv_v(\mathcal B(P_v))$$ for each $\mathcal B\in A\cup B$ and $v\in
S_1$.

(ii) For each irreducible term $p_i(t)$ of the polynomial
$P^{(1)}(t)$, there exists a place $v_i$ such that $p_i(\lambda)$ is
a uniformizer of $\frak o_{k_{v_i}}$ and  is a unit in $\frak o_{k_v}$ if $v\not \in S_1$
and $v\neq v_i$.

For each $\mathcal B\in A\cup B$, by (\ref{Ha:3.1}) and point (i) we have
\begin{equation}\label{sum:1}
0=\sum_{v\in S_1}inv_v(\mathcal B(P_v))=\sum_{v\in S_1}inv_v(\mathcal B(P_v')).
\end{equation}
In particular, for each $\mathcal B=(p_i(\lambda),\chi)\in A$, we
have
$$\sum_{v\in S_1}inv_v(p_i(\lambda),\chi)=0.$$
Since the sum of all local invariants $inv_v$ of $(p_i(\lambda),\chi)$
over $k$ vanishes (global class field theory), we deduce
\begin{equation}\label{equ:S-1}\sum_{v\not\in S_1}inv_v(p_i(\lambda),\chi)=0.
\end{equation}
We have $inv_{v}(p_i(\lambda),\chi)=0$ for $v\not \in S_1$ and
$v\neq v_i$, since $p_i(\lambda)$ is a unit at $v$ by point (ii)
above. Then (\ref{equ:S-1}) implies
\begin{equation}\label{equ:split-1}
inv_{v_i}(p_i(\lambda),\chi)=0.
\end{equation}

Since $p_i(\lambda)$ is a uniformizer at $v_i$ and $\chi$ runs
through all characters of $\Gal(K/k)$ and $K/k$ is abelian, we know $K/k$ is totally
split at $v_i$ by (\ref{equ:split-1}). Therefore the fibre $W_\lambda$ contains a
$k_{v_i}$-point $P_{v_i}'$ for all places $v_i$. Obviously the fibre
$\mathcal W_\lambda$ contains an $\frak o_{k_v}$-point $P_v'$ for all places $v\not
\in S_1$ and $v\neq v_i$ for any $i$, since all $p_i(\lambda)$ are units. So the
fibre $W_\lambda$ contains such a $k_v$-point $P_v'$ for all places $v$
of $k$.

In the following we will show  $inv_v(\mathcal B(P_v'))=0$ for $v\not \in S_1$ and $\mathcal B\in B$, hence \begin{equation}\label{Br-sum:1}
\sum_{v\in \Omega_k}inv_v(\mathcal B(P_v'))=\sum_{v\in S_1}inv_v(\mathcal
B(P_v'))=0
\end{equation}
by (\ref{sum:1}).
Let $v\not \in S_1$ and $v\neq v_i$ for $1\leq i\leq m$. We have $inv_v(\mathcal B(P_v'))=0$ since $\mathcal B$ can be extended to $\mathcal W$ and $P_v' \in \mathcal W(\frak o_{k_v})$.
 If $v$ is some $v_i$, then $K/k$ is totally split at $v$ by the above arguments. Since $\Br(X_K)=\Br_0(X_K)$ by Lemma \ref{Brauer-trivial}, then the image of $\mathcal B$ in $\Br(X_{k_{v}})$ is trivial by the choice of $S$ ($S_1\supset S$), hence $inv_v(\mathcal B(P_v'))=0$.

In the following, we will show the image of $B$ by the induced map $\Br(X)/\Br(k)\rightarrow \Br(X_\lambda)/\Br(k)$ is
surjective, then (\ref{Br-sum:1}) implies there is no Brauer-Manin obstruction on $W_\lambda$ for $(P_v')_{v\in \Omega_k}$, hence we can find  a $k$-point on $W_\lambda$
to approximate $(P_v')_{v\in S}$ by the property of the
principal homogeneous space of tori (\cite[Theorem 8.12]{San} or \cite[Theorem 5.2.1]{Sko01}).

The map $H^1(k,\Pic(\overline X))\rightarrow
\cyr{X}^2_\omega(\widehat{T})_{P}$ in the sequence (\ref{bs}) is
induced by the map $\Pic(\overline X)\to \Pic(X_{\bar{\eta}})$, where $\eta$ is the
generic point of $\P^1_k$.
We have the following commutative diagram
$$
\xymatrix @R=15pt{
\Pic(X_{\bar k[t]_{(t-\lambda)}}) \ar[r] \ar@{=}[d] & \Pic(X_{\bar k(t)}) \ar[d]^{i}\\
\Pic(X_{\bar k[t]_{(t-\lambda)}}) \ar[r] & \Pic(\overline {X_\lambda}). }
$$
Therefore we have the following commutative diagram
$$
\xymatrix @R=15pt{
\Pic(\overline{X}) \ar[r] \ar@{=}[d] & \Pic(X_{\bar k(t)}) \ar[d]^{i}\\
\Pic(\overline{X}) \ar[r] & \Pic(\overline {X_\lambda}). }
$$
By Lemma 2.1(\cite{CHS1}), we know the morphism $i :\Pic(X_{\bar k(t)}) \rightarrow \Pic(\overline {X_\lambda})$ is an isomorphism as $\Gamma_k$-module. Since $\Pic(X_{\bar\eta})$ is torsion-free, we have the isomorphism
$$H^1(k(t), \Pic(X_{\bar\eta}))\cong H^1(k, \Pic(X_{\bar k(t)}))\cong H^1(k,\Pic(\overline {X_\lambda})).$$
Therefore we have the following commutative diagram
$$
\xymatrix @R=15pt{
H^1(k,\Pic(\overline{X})) \ar[r] \ar@{=}[d] & H^1(k(t),\Pic(X_{\bar \eta})) \ar[d]^{\cong}\\
H^1(k,\Pic(\overline{X})) \ar[r] & H^1(k,\Pic(\overline {X_\lambda})),}
$$
hence the image of $B$ by the induced map $\Br(X)/\Br(k)\rightarrow \Br(X_\lambda)/\Br(k)$ is
$\Br(X_\lambda)/\Br(k)$.
\end{proof}

\begin{rem*} \begin{itemize}
\item The condition $\cyr{X}_\omega^2(\widehat T)_P=\cyr{X}_\omega^2(\widehat T)$ is equivalent to the condition that the natural morphism $$\Br(X)/\Br(k)\rightarrow \Br(X_\eta)/\Br(k(\eta))$$
is surjective, where $X$ is the CTHS partial compactification with the projection $p: X\rightarrow \A^1$, and $\eta$ is the generic point of $\A^1$.
\item Many similar results to Theorem \ref{rational-1} have been proved for more general varieties which have a fibration to $\P^1$ , see \cite{CL, Liang3,Sm}.
\end{itemize}
\end{rem*}

As a direct application of this theorem, we have the following
corollary.
\begin{corollary} \label{examples}
Let $n$ be a positive integer and let $K/k$ be an abelian extension with $\Gal(K/k)=\Z/n\times \Z/n$.
Let $P(t)$ be an irreducible polynomial over k and $L=k[t]/(P(t))$. Assume that $L$ contains a cyclic subfield of $K$  with degree $n$.
Let $V$ be the smooth affine variety
over $k$ defined by $$N_{K/k}(\Xi)=P(t).$$ Assume Schinzel's
hypothesis holds, then the Brauer-Manin obstruction to the Hasse
principle and weak approximation is the only obstruction for any
smooth proper model of $V$.
\end{corollary}
\begin{proof} Let $L=k[t]/(P(t))$ and $T$ the torus $R^1_{K/k}(\G_m)$. By the exact sequence $$0\rightarrow \Z \rightarrow \Z[K/k]\rightarrow \widehat T\rightarrow 0,$$
we have $$\cyr{X}^2_\omega(\widehat T)=H^3(K/k,\Z) \text{ and }\cyr{X}^2(\widehat T)_P=\Ker[H^3(K/k,\Z)\rightarrow H^3(K/L\cap K,\Z)].$$
Since $K/L\cap K$ is cycic, we have $$H^3(K/L\cap K,\Z)=H^1(K/L\cap K,\Z)=0,$$
 hence $\cyr{X}^2(\widehat T)_P=\cyr{X}^2(\widehat T)$.
\end{proof}

In the following theorem the condition
$\cyr{X}_\omega^2(\widehat T)_P=\cyr{X}_\omega^2(\widehat T)$ is not in general fulfilled.

\begin{thm}\label{rational-2}
Let $k$ be a number field and $P(t)$ a nonzero polynomial over $k$. Assume
Schinzel's hypothesis holds. Then the Brauer-Manin obstruction to
the Hasse principle and weak approximation for rational points is the only obstruction
for any smooth proper model of the variety over $k$ defined by the
equation
$$(x_1^2-ax_2^2)(y_1^2-by_2^2)(z_1^2-abz_2^2)=P(t),$$
where $a,b\in k^\times$.
\end{thm}
\begin{proof}  It is sufficient to prove the statement for any given model. Let $V$ be the smooth locus of the affine $k$-variety defined by
$$(x_1^2-ax_2^2)(y_1^2-by_2^2)(z_1^2-abz_2^2)=P(t).$$ Let $Y$ be a
smooth compactification of $V$ with a projection $p: Y\rightarrow
\P^1_k$. If one of the three numbers $a,b,ab$ is a square in $k^\times$,
this theorem is obvious since $Y$ is rational. Then we only need to consider the case all
numbers $a,b,ab$ are not squares in $k^\times$.

We assume that $Y$ has points in all completions of $k$, and we are
given a finite set $S$ of places of $k$ containing all archimedean places, and points $P_v\in Y(k_v)$
for $v\in S$. We assume that there is no Brauer-Manin obstruction to
weak approximation for $(P_v)_{v\in S}$.

By a similar argument as
in the proof of Theorem \ref{rational-1}, we can replace $P_v$ by a point close to it in an open subset of $Y$ for each finite place in $S$ and "move $P_v$ to the infinite point" for each archimedean place $v$ , then we get an open smooth
affine subvariety $U$ of $Y$ defined by
$$(x_1^2-ax_2^2)(y_1^2-by_2^2)(z_1^2-abz_2^2)=P^{(1)}(t)\neq 0,$$
such that $P_v$ is contained in $U$ for all $v\in S$ and $P_v$ is very close to the smooth fiber of $U$ at the infinite point for all archimedean places.

Let $U_0$ be the open subvariety of $\A^1$ defined by $P^{(1)}\neq 0$.
Now we only need to look for a point $Q\in U_0(k)$ with the associated
coordinate $\lambda\in k$, such that $\lambda$ is very close to each
$p(P_v)$ for any finite place $v\in S$, $\lambda$ is big enough at each archimedean place of $k$ ($Q$ is close enough to the infinite point), and such that the
fibre $U_\lambda$ has a $k$-rational point.

Let $P^{(1)}(t)=cp_1(t)^{e_1}\cdots p_m(t)^{e_m}$, where $p_i(t)$ is irreducible over $k$. Let
$$A=\{(p_i(t),b)\in \Br(U): 1\leq i \leq m\}\cup \{(x_1^2-ax_2^2,b)\}.$$ We
know $(x_1^2-ax_2^2,b)$ is the unique generator of the unramified Brauer group
of the smooth fibre of $Y$ (see \cite[Theorem 4.1]{CT11}). Obviously $A\subset \Br(U)$.

We enlarge
$S$ so that it contains all the original places at which we want to
approximate and that it also contains the places associated in
Hypothesis $H_1$ (\cite[p. 71]{CTSD94}) to the polynomials of
$P^{(1)}(t)$ and all ramified places of $K/k$, where
$K=k(\sqrt{a},\sqrt{b})$.

There is no Brauer-Manin obstruction to weak approximation for
$(P_v)_{v\in S}$. According to Harari's formal lemma (see
\cite{Ha94}), we may find a finite set $S_1$ of places of $k$
containing $S$, and points $P_v\in U(k_v)$, $v\in S_1$, and which
extend the given family
$$P_v \in U(k_v),v\in S,$$ such that for each $\mathcal B\in
A$
\begin{equation}\label{Ha:3.2}
\sum_{v\in S_1}inv_v(\mathcal B(P_v))=0.
\end{equation}

Applying Hypothesis $(H_1)$ (\cite[Proposition 4.1]{CTSD94}),
we thus find $\lambda\in k$ close enough to each $\lambda_v=p(P_v)$
for the finite places $v\in S_1$, $\lambda$ integral away from
$S_1$, and $\lambda$ as large as we wanted at the archimedean places, such
that:
\begin{enumerate}[(i)]
\item The fibre $U_\lambda$ of $p$ contains a $k_v$-point $P_v'$ which
is as close as we wish to $P_v$ for all places $v\in S_1$, and such
that
$$inv_v(\mathcal B(P_v'))=inv_v(\mathcal B(P_v))$$ for each $\mathcal B\in A$ and $v\in
S_1$.

\item For each irreducible term $p_i(t)$ of
$P^{(1)}(t)$, there exists a place $v_i$ such that $p_i(\lambda)$ is
a uniformizer at $v_i$ and  is a unit at $v$ if $v\not \in S_1$
and $v\neq v_i$.
\end{enumerate}

If $v\not \in S_1$, then $K/k$ is unramified at $v$. Then one of
$a,b,ab$ is a square in $k_v^\times$. Then the fibre $U_\lambda$
contains a $k_{v}$-point $P_v'$ for all places $v\not \in S_1$.


Let $\mathcal B=(x_1^2-ax_2^2,b)$. In the following we will show $inv_v(\mathcal
B(P_v'))=0$ for any $v\notin S_1$:
\begin{enumerate}[a)]
\item Suppose $v\neq v_i$ for $1\leq i\leq m$.
\begin{enumerate}[1)]
\item If $k(\sqrt{b})/k$ is split at $v$, then $inv_v(\mathcal
B(P_v'))=0$.

\item If $k_v(\sqrt{b})=k_v(\sqrt{a})$ is inert over $k_v$,
then $inv_v(\mathcal B(P_v'))=0$ obviously.

\item If $k_v(\sqrt{b})/k_v$ is inert
and $k_v(\sqrt{a})/k_v$ is split, then
$$(y_1^2-by_2^2,b)_v=(z_1^2-abz_2^2,b)_v=0.$$ One has
$$inv_v(\mathcal
B(P_v'))=(x_1^2-ax_2^2,b)_v=(P^{(1)}(\lambda),b)_v=0$$ since
$P^{(1)}(\lambda)$ is a unit at $v$.
\end{enumerate}

\item Suppose $v$ is some $v_i$. For each $(p_i(t),b)\in A$, by (\ref{Ha:3.2}) and point (i), we have
$$\sum_{v\in S_1}inv_v(p_i(\lambda),b)=0.$$
By global class field theory, we have
\begin{equation} \label{equ:S-2}
\sum_{v\not\in S_1}inv_v(p_i(\lambda),b)=0.
\end{equation}
By point (ii) above, we have $inv_{v}(p_i(\lambda),b)=0$ for
$v\not \in S_1$ and $v\neq v_i$, since $p_i(\lambda)$ is a unit at
$v$. Then (\ref{equ:S-2}) implies $inv_{v_i}(p_i(\lambda),b)=0.$ Since $p_i(\lambda)$ is a
uniformizer at $v_i$, we have $k(\sqrt{b})/k$ is totally split at
$v_i$, hence $inv_{v_i}(\mathcal B(P_{v_i}'))=0$.
\end{enumerate}

By above arguments, we have
$$\sum_{v\in \Omega_k}inv_v(\mathcal B(P_v'))=\sum_{v\in S_1}inv_v(\mathcal
B(P_v'))=0$$
by (\ref{Ha:3.2}) and point (i) for the last equality. Note that $\mathcal B$ generates the unramified Brauer group $\Br_{ur}(U_\lambda)/\Br_0(U_\lambda)$, there is no Brauer-Manin obstruction on
$U_\lambda$ for $(P_v')_{v\in \Omega_k}$. There is a $k$-point on $U_\lambda$ which
is close enough to $(P_v')_{v\in S}$ by \cite[Theorem 8.12]{San}.
\end{proof}

In the above we considered the case that the extension $K/k$ is
abelian. In the following we will consider the case that $K/k$ is
non-abelian. The original idea owes to Colliot-Th\'el\`ene and the
further development owes to Wittenberg.

\begin{thm}\label{rational-3}
Let $k$ be a number field and  $K/k$ non-Galois extension of prime degree $p$. Let $K^{cl}$ be the Galois closure of $K/k$ and $F/k$ the maximal abelian subextension of $K^{cl}/k$. Let $P(t)$ be a nonzero polynomial over $k$. For any irreducible factor $p(t)$ of $P(t)$, let $L=k[t]/(p(t))$. We suppose
\begin{equation}\label{assumption:P}
F\not \subset L,\text{ or } F\subset L\text{ and } K.L/L \text{ is abelian};
\end{equation}
e.g., if $p=3$, any nonzero polynomial $P(t)$ satisfies (\ref{assumption:P}).
Assume Schinzel's hypothesis holds. Then the Brauer-Manin obstruction to the Hasse principle and
weak approximation for rational points is the only obstruction for any smooth proper
model of the variety over $k$ defined by the equation
$$N_{K/k}(\Xi)=P(t).$$
\end{thm}

Before the proof of Theorem \ref{rational-3}, we need a lemma, which is similar to \cite[Lemma 4.8]{HSW}.
\begin{lem}\label{lemma:local} If $v$ is a place of $k$ which is unramified in $K^{cl}/k$ and not totally split in $F/k$. Then the equation $N_{K/k}(\Xi)=a$ is solvable over $k_v$ for any $a\in k_v^\times$.
\end{lem}
\begin{proof} Write $K\otimes_k k_v:=K_{w_1}\oplus\cdots\oplus K_{w_s}$. First we will show $s>1$.

If $s=1$, then $K\otimes_k k_v=K_w$ is a field extension of $k_v$ of degree $p$. By the assumption, $v$ is not totally split in $F/k$, therefore $F\otimes_k k_v= F_{u_1}\oplus\cdots \oplus F_{u_t}$ with $F_{u_1}\neq k_v$. Let $d:=[F_{u_1}:k_v]>1$. Since $\Gal(K^{cl}/k)$ is a subgroup of the symmetric group  ${\mathcal S}_p$, we have $d<p$, hence $K_w$ and $F_{u_1}$ are linearly disjoint over $k_v$. Therefore the
Frobenius at $v$ in $\Gal(K^{cl}/k)$ is an element of order divisible by $dp$. However,
$\mathcal S_p$ contains no such elements, so the case s = 1 is impossible.

Let $d_i:=[K_{w_i}:k_v]$. Since $p = d_1 + ... + d_s$ is a prime number and $s>1$, then $\text{gcd}(d_1,\dots,d_s)=1$. Therefore there exist integers
$n_1, \dots, n_s$ such that $1 = n_1 d_1 + \dots + n_s d_s$. It follows that
$$a=\prod_{i=1}^s N_{K_{w_i}/k_v}(a^{n_i})\in N_{K/k}(K\otimes_k k_v),$$
so we are done.
\end{proof}

\begin{proof}[The proof of Theorem \ref{rational-3}]
Let $V$ be the smooth locus of the affine $k$-variety defined by
$N_{K/k}(\Xi)=P(t)$.
Let $Y$ be a smooth compactification of $V$
with a projection $p: Y\rightarrow \P^1_k$ defined by
$(\Xi,t)\mapsto t$.

%
%

We assume that $Y$ has points in all completions of $k$, and we are
given a finite set $S$ of places of $k$ containing all archimedean places, and points $P_v\in Y(k_v)$
for $v\in S$. We assume that there is no Brauer-Manin obstruction to
weak approximation for $(P_v)_{v\in S}$.

By a similar argument as
in the proof of Theorem \ref{rational-1}, we can replace $P_v$ by a point close to it in an open subset of $Y$ for each finite place in $S$ and "move $P_v$ to the infinite point" for each archimedean place $v$ , $i.e.$, we get an open smooth
affine subvariety $U$ of $Y$ defined by
$$N_{K/k}(\Xi)=P^{(1)}(t) \neq 0,$$
such that $P_v$ is contained in $U$ for all $v\in S$ and $P_v$ is very close to the smooth fiber of $U$ at the infinite point for all archimedean places.

Let $U_0$ be the open subvariety of $\A^1$ defined by $P^{(1)}(t)\neq 0$.
We are now looking for a point $Q\in U_0(k)$ with associated
coordinate $\lambda\in k$, such that $\lambda$ is very close to each
$Q_v$ for $v\in S$, $v$ finite, $\lambda$ is big enough at each archimedean place ($Q$ is close enough to the infinite point), and such that the
fibre $U_\lambda$ has a $k$-rational point.

Let $P^{(1)}(t)=cp_1(t)^{e_1}\cdots p_m(t)^{e_m}$, where $p_i(t)$ is monic and irreducible over $k$. Let $L_i=k[t]/(p_i(t))$. By our assumption (\ref{assumption:P}) for $P(t)$, it is easy to verify that
\begin{equation}\label{assumption:P1}
F\not \subset L_i,\text{ or } F\subset L_i\text{ and } K.L_i/L_i \text{ is abelian}.
\end{equation}
Define
$$T_1=\{1\leq i\leq m:
F\not \subset L_i\}\text{ and }T_2=\{1\leq i\leq m:
F\subset L_i\}.$$
For $i\in T_1$, we choose a nontrivial character $\chi_i$ of $\Gal(F/k)$, which is nontrivial on the subgroup $\Gal(F/F\cap L_i)$.
Denote $$\begin{aligned}A_1&=\{(p_i(t),\chi_i)\in \Br(U): i\in T_1\},\\
A_2&=\bigcup_{i\in T_2}\{\Cor_{L_i/k}(t-\eta_i,\psi)\in \Br(U): \psi\in \Hom(\Gal(K.L_i/L_i),\Q/\Z)\},\end{aligned}$$
where $\eta_i$ is a root of $p_i(t)$ in $L_i$.

Let $S'$ be a finite set of places of $k$ containing all the archimedean places and all the bad
finite places in sight: finite places where one $p_i(t)$ is not
integral, finite places where all $p_i(t)$ are integral but the
product $\prod_{i}p_i(t) $ does not remain separable when reduced
modulo $v$, places ramified in $K^{cl}/k$.

We enlarge $S$ such that $S\supset S'$ and that $S$ contains all the original places at which
we want to approximate, and that $S$ also contains the places
associated in Hypothesis $H_1$ (\cite[p. 71]{CTSD94}) to the
polynomials of $P^{(1)}(t)$ and all ramified places of $K^{cl}/k$.
There is no Brauer-Manin obstruction to weak approximation for
$(P_v)_{v\in S}$. According to Harari's formal lemma (see
\cite{Ha94}), we may find a finite set $S_1$ of places of $k$
containing $S$, and points $P_v\in U(k_v)$, $v\in S_1$, which extend
the given family
$$P_v \in U(k_v),v\in S,$$ such that for each $\mathcal B\in
A_1\cup A_2$
\begin{equation*}
\label{Ha:3.3}\sum_{v\in S_1}inv_v(\mathcal B(P_v))=0.
\end{equation*}

We claim that for any $i \in T_1$, we can find a place $v_i\notin S_1$ of $k$ and a point $P_{v_i}$ in $U(k_{v_i})$ such that
$$inv_{v_i}(p_i(\lambda_{v_i}),\chi_i)\neq 0 \text{ and } inv_{v_i}(\mathcal B(P_{v_i}))=0
$$ for any $\mathcal B\in A_1\cup A_2$ and $\mathcal B\neq
(p_i(t),\chi_i)$, where $\lambda_{v_i}$ is the coordinate of the image of $P_{v_i}$ in $\A^1$. Therefore we extend $S_1$ to $S_1\cup \{v_i: i\in T_1\}$,
then we have
\begin{equation}\label{a1}
\sum_{v\in S_1}inv_v(\mathcal B(P_v))\neq 0 \text{ for } \mathcal B\in
A_1
\end{equation}
and
\begin{equation}\label{a2}\sum_{v\in S_1}inv_v(\mathcal B(P_v))=0 \text{ for }
\mathcal B\in A_2.\end{equation}

We now prove this claim. For $i\in T_1$, let $E_i=F.L_i^{cl}$, where $L_i^{cl}$ is the Galois closure of $L_i/k$. Let $\Res_{k/L_i}(\chi_i)$ be the restriction of $\chi_i$ to $\Gal(F.L_i/L_i)$, which is a nontrivial character by the choice of $\chi_i$. Obviously there exists $g\in \Gal(E_i/L_i)$ with $\Res_{k/L_i}(\chi_i)(\bar g)\neq 0$, where $\bar g$ is the image of $g$ in the quotient group $\Gal(F.L_i/L_i)$.
There are infinitely many places $v$ of $k$ such that $g$ is contained in the
conjugation class of the Frobenius of $v$ by Chebotarev's density
theorem. We can choose such a place $v_i$ with $v_i\not \in S$.
Therefore $p_i(t)=\pi_{v_i}$ has a solution $\lambda_{v_i}\in \frak
o_{k_{v_i}}$ and $p_j(\lambda_{v_i}) \in \frak o_{k_{v_i}}^\times$
for $j\neq i$, where $\pi_{v_i}$ is a uniformizer of $k_{v_i}$. Since $v_i$ is not totally split in $E/k$, by Lemma \ref{lemma:local}, $U_{\lambda_{{v_i}}}$ has a $k_{v_i}$-point $P_{v_i}$. Therefore we have
$$inv_{v_i}(p_i(\lambda_{v_i}),\chi_i)\neq 0 \text{ and } inv_{v_i}(\mathcal B(P_{v_i}))=0
$$ for any $\mathcal B\in A_1$ and $\mathcal B\neq
(p_i(t),\chi_i)$. Suppose $\mathcal B\in A_2$, since $p_j(\lambda_{v_i})$ is a unit, we also have $inv_{v_i}(\mathcal B(P_{v_i}))=0$.

Now apply Hypothesis $(H_1)$ (\cite[Proposition
4.1]{CTSD94}), we thus find $\lambda\in k$ close enough to each
$\lambda_v=p(P_v)$ for the finite places $v\in S_1$, $\lambda$
integral away from $S_1$, and $\lambda$ as large as we wanted at all
archimedean places, such that:

(i) The fibre $U_\lambda$ of $p$ contains a $k_v$-point $P_v'$ which
is as close as we wish to $P_v$ for all places $v\in S_1$, and such
that
$$inv_v(\mathcal B(P_v'))=inv_v(\mathcal B(P_v))$$ for each $\mathcal B\in A_1\cup A_2$ and $v\in
S_1$.

(ii) For each irreducible term $p_i(t)$ of the polynomial
$P^{(1)}(t)$, there exists a place $v_i'$ such that $p_i(\lambda)$
is a uniformizer at $v_i'$ and  is a unit at $v$ if $v\not \in
S_1$ and $v\neq v_i'$.

In the following we will show that $U_\lambda$ also contains a $k_v$-point for $v \notin S_1$.
Since $\Br(U_\lambda)=\Br_0(U_\lambda)$ (\cite[Proposition 9.1]{CT/Sa87}), there is a
$k$-point on $U_\lambda$ which is close enough to $(P_v')_{v\in S}$
by \cite[Theorem 8.12]{San}.

Suppose $v\notin S_1$ and $v\neq v_i'$ for $1\leq i\leq m$. Since $P^{(1)}(\lambda)$ is a unit in $k_v$, $U_\lambda$ contains a $k_v$-point.

Let $i\in T_1$.  For each $\mathcal B\in A_1$, by (\ref{a1}) and point (i), we have
$$0\neq \sum_{v\in S_1}inv_v(\mathcal B(P_v))=\sum_{v\in S_1}inv_v(\mathcal B(P_v')).$$
Then we have
$$\sum_{v\in S_1}inv_v(p_i(\lambda),\chi_i)\neq 0.$$
By global class field theory, we have
\begin{equation}\label{equ:S-3}
\sum_{v\not\in S_1}inv_v(p_i(\lambda),\chi_i)\neq 0.
\end{equation}
By point (ii) above, we have $inv_{v}(p_i(\lambda),\chi_i)=0$ for
$v\not \in S_1$ and $v\neq v_i'$, since $p_i(\lambda)$ is a unit at
$v$. Then (\ref{equ:S-3}) implies
$$inv_{v_i'}(p_i(\lambda),\chi_i)\neq 0.$$ Since $p_i(\lambda)$ is a
uniformizer at $v_i'$, we have $F/k$ is not totally split at $v_i'$. By Lemma \ref{lemma:local}, $U_\lambda$ has a
$k_{v_i'}$-point for $i\in T_1$.

Let $i\in T_2$. For any $\mathcal B=\Cor_{L_i/k}(t-\eta_i,\psi)\in A_2$,  by equation (\ref{a2}) and class field theory, similarly we have
$$0= \sum_{v\not \in S_1}inv_v(\Cor_{L_i/k}(\lambda-\eta_i,\psi))=inv_{v_i'}(\Cor_{L_i/k}(\lambda-\eta_i,\psi)),$$
which is equivalent to:
\begin{equation}\label{equ:inv}
\sum_{w\in \Omega_{L_i}, w\mid v_i'}inv_w(\lambda-\eta_i,\psi)=0.
\end{equation}
Since
$ord_{v_i'}(p_i(\lambda))=1$, in the decomposition of $L_i\otimes_k k_{v_i'}$ into a product of local fields $k_{i,w}$, $\lambda-\eta_i$ goes to a unit into all local fields $k_{i,w}$ but one, call it $k_{i,w_i}$, which is of degree one over $k_{v_i'}$ in which $\lambda-\eta_i$ becomes to a uniformizer. 
In (\ref{equ:inv}), all terms but one then vanish. Thus the remain one $inv_{w_i}(\lambda-\eta_i,\psi)$ also vanishes. Since $\lambda-\eta_i$ is a uniformizer in $L_{i,w_i}$ and $\psi$ runs through all characters of $\Gal(K.L_i/L_i))$, $K.L_i/L_i$ is totally split at $w_i$. Therefore $K$ can embed into $L_{i,w_i}=k_{v_i'}$, hence $U_\lambda$ has a $k_{v_i'}$-point
for $i\in T_2$.
\end{proof}

\begin{rem*}If $k=\Q$ and $P(t)$ is  a product of (possibly repeated) linear factors over
$\Q$, without assuming Schinzel's hypothesis, similar results to \ref{rational-2} and \ref{rational-3} have been proved recently (see \cite{BM,HSW}).
\end{rem*}

\section{Brauer-Manin properties for zero-cycles of degree 1}\label{zerocyclesection}

This section is devoted to the proof, for zero-cycles of degree 1,
of  unconditional versions of the theorems in Section
\ref{Schinzelsection}, using Salberger's device, as in
\cite{CTSSD98}. The same general comments as made in the beginning
of the previous section  may be made here, some related works were also performed in \cite{Liang1,Liang2,Wit}. In particular most of the
results in the present section are not covered by \cite[Theorem 4.1]{CTSSD98}.

\begin{thm}\label{cycle-1}
Let $k$ be a number field and $K/k$ an abelian extension. Let $P(t)$
be a nonzero polynomial over $k$.  Let $V$ be the smooth locus of the affine
$k$-variety defined by
$$N_{K/k}(\Xi)=P(t).$$
Let $T=R^1_{K/k}(\G_m)$. Suppose $\cyr{X}_\omega^2(\widehat
T)_P=\cyr{X}_\omega^2(\widehat T)$ (see
Section 1 for definition). If there is no Brauer-Manin
obstruction to the existence of a zero-cycle of degree 1 on a smooth
proper model $V^c$ of $V$, then there is a zero-cycle of degree 1 on
$V^c$ (defined over $k$).
\end{thm}
\begin{proof}
Let $U$ be the smooth affine variety over $k$ defined by
$$N_{K/k}(\Xi)=P(t)\neq 0.$$
Let $U^c$ be a smooth compactification of $U$ with a projection $p:
U^c \rightarrow \P^1_k$.

Let $X$ be the \emph{CTHS} partial compactification (see Section 1 for definition) of $V$. By the assumption, we know $\cyr{X}^2_\omega (\widehat T
)_P=\cyr{X}^2_\omega (\widehat T )$. Then we can choose a finite
subset $B\subset \Br(X)$, such that the image of $B$ by the
composite map $$\Br(X) \rightarrow H^1(k,\Pic(\overline X))\rightarrow
\cyr{X}^2_\omega (\widehat T)$$ is $\cyr{X}^2_\omega
(\widehat T)$ by the sequence (\ref{bs}).

Let $P(t)=cp_1(t)^{e_1}\cdots p_m(t)^{e_m}$, where $p_i(t)$ is a monic irreducible polynomial over $k$.
Choose $\omega_1,\cdots,\omega_n\in \frak o_K$ to be a basis of $K$ over $k$, where $n=[K:k]$.
Choose $S$ to be a finite set of places of $k$ containing all archimedean places such that $P(t) \in \frak o_S[t]$ and $\omega_1,\cdots,\omega_n$ generate $\frak o_K\otimes_{\frak o_k} {\frak o_{k_v}}$ over $\frak o_{k_v}$ for any $v \notin S$. Let $\mathcal U \subset \A_{\frak o_S}^{n+2}$ be the smooth affine integral model of $U$ over $\frak o_S$ defined by
$$N_{K/k}(x_1 \omega_1+\cdots+x_n\omega_n)=P(t) \text{ and } y\cdot P(t)=1.$$
We enlarge $S$ such that for any $\mathcal B\in B$, $\mathcal B$ can be extended to $\Br(\mathcal U)$.

Let $S_0$ be a finite
set of places of $k$ containing $S$ and all
the bad finite places in sight: finite places where one $p_i(t)$ is
not integral, finite places where all $p_i(t)$ are integral but the
product $\prod_{i}p_i(t) $ does not remain separable when reduced
modulo $v$, places ramified in the extension $K/k$ and and all
places $v$ such that for some $\mathcal B \in B$, $inv_w(j_K(\mathcal B))\neq 0$ where
$w$ is a place of $K$ over $v$ and $j_K: \Br(X) \to \Br(X_K)$, noting that $\Br(X_K)=\Br_0(X_K)$ by Lemma \ref{Brauer-trivial}.

Let $$A=\{(p_i(t),\chi)\in \Br(U): 1\leq i\leq m, \chi \in \Hom(
\Gal(K/k),\Q/\Z)\}.$$ Since we assume that there is no Brauer-Manin
obstruction to the existence of a zero-cycle of degree 1, by an
obvious variant of Harari's result(\cite[Theorem 3.2.2]{CTSD94}) we
may find a finite set $S_1$ of places of $k$ containing $S_0$ and
for each $v\in S_1$ a zero-cycle $z_v$ of degree 1 with support in
$U\times_{k}k_{v}$ such that
\begin{equation*}
     \sum_{v\in S_1}inv_{v}(\left<\mathcal B,z_{v}\right>)=0 \text{ for all } \mathcal B\in A\cup
     B.
\end{equation*}

Let $s$ be the least common multiple of the orders of $\mathcal B\in
A\cup B$. Let us write the zero-cycle $z_v$ as
$z_{v}^{+}-z_{v}^{-}$, with $z_{v}^{+}$ and $z_{v}^{-}$ effective
cycles. Let $z_{v}^{1}$ be the effective cycle
$z_{v}^{+}+(ns-1)z_{v}^{-}$. We have $z_{v}=z_{v}^{1}-nsz_{v}^{-}$,
hence $\left<\mathcal B,z_v\right>=\left<\mathcal
A,z_{v}^{1}\right>$ since each $\mathcal B$ is killed by $s$. We
thus have
\begin{equation*}
     \sum_{v\in S_1}inv_{v}(\left<\mathcal B,z_{v}^1\right>)=0 \text{ for all } \mathcal B\in A\cup
     B.
\end{equation*}

Let $N_0$ be a closed point of $U$ such that
$k(N_0)=K$. Similarly $\left<\mathcal B,sN_0\right>=0$. The degree of
$z_{v}^{1}$ is congruent to $1$ modulo $ns$. The cycle $sN_0$ has
degree $ns$. Adding suitable multiples of $sN_0$ to each $z_{v}^{1}$
for $v$ in the finite set $S_1$, then we find effective cycles
$z_{v}^{2}$, all of the same degree $1+Dns$ for some $D>0$, and such
that
\begin{equation}\label{equ:3}
     \sum_{v\in S_1}inv_{v}(\left<\mathcal B,z_{v}^2\right>)=0 \text{ for all } \mathcal B\in A\cup
     B.
\end{equation}
By the implicit function
theorem and the continuity of the evaluation maps of the Brauer group, in (\ref{equ:3}), for each $v \in S_1$, each effective
cycle $z_v^2$ may be assumed to be a sum of distinct closed points
(i.e. there are no multiplicities) whose images under $p_{k_v}:
U_{k_v} \rightarrow \A_{k_v}^{1}$ are also distinct.

We claim that while keeping (\ref{equ:3}) we can moreover assume
that, for each $z_v^{2}$ and each closed point $P$ in the support of
$z_v^{2}$, the field extension map $k_v(f(P)) \subset k_v(P)$ is an
isomorphism. Once more, it is enough to replace $P$ by a suitable
and close enough $k_v(P)$-rational point on $U_{k_v}(P)$: this
follows from \cite[Lemma 6.2.1]{CTSD94}.

Each of the zero-cycles $p(z_v^2)$ is now given by a separable monic
polynomial $G_v[t] \in k_v[t]$ of degree $1+Dns$, prime to $P(t)$
and with the property that the smooth fibres of $p$ above the roots
of $G_v$ have  rational points over their field of definition.
By Krasner's lemma, any monic polynomial $H(t)$ close
enough to $G_v(t)$ for the $v$-adic topology on the coefficients
will be separable, with roots `close' to those of $G_v$. Thus the
fibres above the roots of the new polynomial are still smooth and
still possess rational points  over their field of definition.

An irreducible polynomial $G(t)$ of degree $1+Dns$ defines a closed
point $M$ of degree $1+Dns$ on $\A_k^{1}$. Let $F:=k(M)=k[t]/(G(t))$.
Let $\theta$ be the residue class of $t$ in $F$. We now choose the
irreducible polynomial $G(t)$ as given by
\cite[Theorem 3.1]{CTSSD98} with the field $L$ in Theorem 3.1 contains $K$, and $V'$
($V$ in Theorem 3.1) is the set of places of $k$ at
which $L$ is totally split, such that

(i) $\theta$ is integral at $v\notin S\cup V'$.

(ii) For each place $v\in S_1$, $G(t)$ is close enough to $G_v(t)$, such
that the fibre $U_\theta$ contains an $F_w$-point $P_w$ for
each place $w$ of $F$ over $v$, and such that $\sum_{w\mid v}P_w$ is
`close' enough to $z_v^2$ satisfying
$$\sum_{w\mid v}inv_v(\Cor_{F_w/k_v}(\mathcal B(P_w)))=inv_v(\left<\mathcal B,z_{v}^2\right>)$$ for each $\mathcal B\in A\cup B$.

(iii) For each irreducible term $p_i(t)$ of the polynomial $P(t)$,
there exists a place $w_i$ of $F$ such that $p_i(\theta)$ is a
uniformizer at $w_i$ and  is a unit at $w$ if $w \text { is not
over } S_1\cup V'$ and $w\neq w_i$.

We claim that the fibre $U_\theta/F$ also has a point for $w$ not over $S_1$.

\begin{enumerate}[1)]
\item If the  place $w$ is not over $S_1\cup V'$ and $w\neq w_i$
for all $i$, it is clear that $\mathcal U_\theta$ has a $\frak o_{F_w}$-point $P_w$
since  $\theta$ is integral at $w$ by point (i) and $P(\theta)$  is a unit at $w$ by point (iii).

\item If the  place $w$ is over $V'$,  it is clear that $U_\theta$ has a $F_w$-point $P_w$ since $K.F/F$ is totally
split at $w$.

\item  For each $\mathcal B=(p_i(t),\chi)\in A$, by (\ref{equ:3}) and point (ii), we have
$$\sum_{w\mid v,v\in S_1}inv_v(\Cor_{F_w/k_v}(p_i(\theta),\chi))=0.$$
By global reciprocity law, one has
\begin{equation}\label{equ:cycle-1}
\sum_{w\mid v,v\not\in S_1}inv_v(\Cor_{F_w/k_v}(p_i(\theta),\chi))=0.
\end{equation}
However, $inv_{w}(p_i(\theta),\chi)=0$ if $w$ is not over $S_1$ and
$w\neq w_i$, since either $p_i(\theta)$ is a unit in $F_w$ or $K.F/F$ is totally
split by point (iii)
above. Hence (\ref{equ:cycle-1}) implies
$$inv_{w_i}(p_i(\theta),\chi)=0.$$
Since $\chi$ runs through all characters of
$\Gal(K/k)$ and $p_i(\theta)$ is a
uniformizer at $w_i$ by point (iii), we have $K.F/F$ is totally split at $w_i$. Therefore
the fibre $U_\theta$ contains an $F_{w_i}$-point $P_{w_i}$.
\end{enumerate}

We claim $inv_w(\mathcal
B(P_w))=0$ for any $w$ is not over $S_1$.
\begin{enumerate}[a)]
\item Suppose $w$ is a place of $F$ over $v\not \in
S_1\cup V'$ and $w\neq w_i$ for $1\leq i\leq m$. For any $\mathcal B\in B$, we have $inv_w(\mathcal
B(P_w))=0$ since $\mathcal B$ can be extended to $\mathcal U$ and $P_w \in \mathcal U(\frak o_{F_w})$.

\item Suppose $w$ is some $w_i$ or $w$ is over $V'$.  Then $K.F/F$ is totally split at $w$ by the above argument 3) for $w=w_i$ and the definition of $V'$ for $w$ is over $V'$. Since $\Br(X_K)=\Br_0(X_K)$ by Lemma \ref{Brauer-trivial}, the image of $\mathcal B$ in $\Br(X_{F_{w}})$ is trivial by the choice of $S_0$ ($S_1\supset S_0$), hence $inv_w(\mathcal
B(P_w))=0$.
\end{enumerate}

By (\ref{equ:3}), point (ii) and the above arguments, we
have
$$\sum_{w\in \Omega_F}inv_v(\Cor_{F_w/k_v}(\mathcal B(P_w)))=\sum_{v\in S_1}inv_v(\left<\mathcal B,z_{v}^2\right>)=0.$$
Therefore
$$\sum_{w\in \Omega_F}inv_w(\mathcal B(P_w))=0$$ by the equality $inv_v(\Cor_{F_w/k_v}\mathcal B(P_w))=inv_w(\mathcal
B(P_w))$.
 Since $[F:k]$ and
$[K:k]$ are relatively prime, we have $F\cap K=k$. Therefore the
natural map $\Br(T^c)/\Br_0(T^c)\rightarrow \Br(T^c_F)/\Br_0(T^c_F)$
is an isomorphism. By a similar argument as the last part in the proof of Theorem \ref{rational-1}, $B$ generates the unramified Brauer group
$\Br_{ur}(U_\theta)/\Br_0(U_\theta)$. Therefore there is no
Brauer-Manin obstruction on $U_\theta$ for $(P_w)_{w\in \Omega_F}$, hence
$U_\theta/F$ possesses an $F$-point by \cite[Theorem 8.12]{San}.
\end{proof}

With the help of Salberger's device, by similar argument as above we
have the following result which corresponds to Theorem
\ref{rational-2}.

\begin{thm} \label{cycle-2}
Let $k$ be a number field and $P(t)$ a nonzero polynomial over $k$. Let $V$
be the smooth locus of the affine $k$-variety defined by
$$(x_1^2-ax_2^2)(y_1^2-by_2^2)(z_1^2-abz_2^2)=P(t)$$
where $a,b\in k^\times$. If there is no Brauer-Manin obstruction to the
existence of a zero-cycle of degree~1 on a smooth proper model $V^c$
of $V$, then there is a zero-cycle of degree 1 on $V^c$ (defined
over $k$).
\end{thm}
\begin{proof} Let $U$ be the smooth affine variety over $k$ defined by
$$(x_1^2-ax_2^2)(y_1^2-by_2^2)(z_1^2-abz_2^2)=P(t)\neq 0.$$
Let $U^c$ be a smooth compactification of $U$ with a projection $p:
U^c \rightarrow \P^1_k$. If one of the three numbers $a,b,ab$ is a
square in $k^\times$, this theorem is obvious since $Y$ is rational. Then we only need to
consider the case all numbers $a,b,ab$ are not
squares in $k^\times$.

Write $P(t)=cp_1(t)^{e_1}\cdots p_m(t)^{e_m}$, where $p_i(t)$ is monic and irreducible over $k$. Let $S_0$ be a finite
set of places of $k$ containing all the archimedean places and all
the bad finite places in sight: finite places where one $p_i(t)$ is
not integral, finite places where all $p_i(t)$ are integral but the
product $\prod_{i}p_i(t) $ does not remain separable when reduced
modulo $v$, places ramified in the extension $K/k$ where
$K=k(\sqrt{a},\sqrt{b})$.

Let $$A=\{(p_i(t),b)\in \Br(U): 1\leq i \leq m\}\cup
\{(x_1^2-ax_2^2,b)\}.$$ We know $(x_1^2-ax_2^2,b)$ is the unique
generator of the unramified Brauer group of the smooth fibre of $U$ (see \cite[Theorem 4.1]{CT11}).

Since we assume that there is no Brauer-Manin obstruction to the
existence of a zero-cycle of degree 1, by an obvious variant of
Harari's result(\cite[Theorem 3.2.2]{CTSD94}) we may find a finite
set $S_1$ of places of $k$ containing $S_0$ and for each $v\in S_1$
a zero-cycle $z_v$ of degree 1 with support in $U\times_{k}k_{v}$
such that
\begin{equation}\label{Ha:4.2}
     \sum_{v\in S_1}inv_{v}(\left<\mathcal B,z_{v}\right>)=0 \text{ for all } \mathcal B\in A.
\end{equation}
By the similar argument as in the proof of Theorem \ref{cycle-1},
for each $v \in S_1$,  we may assume that:
\begin{enumerate}[1)]
\item Each cycle $z_v$ is effective and has the same degree $D$, where
$D$ is odd.
\item Each cycle $z_v$ is a sum of distinct closed points $P_v$ whose
images under $p_{k_v}: U_{k_v} \rightarrow \A_{k_v}^{1}$ are also
distinct, and the field extension map $k_v(p(P_v))\subset k_v(P_v)$
is an isomorphism.
\end{enumerate}

Each of the zero-cycles $p(z_v)$ is now given by a separable monic
polynomial $G_v[t] \in k_v[t]$ of degree $D$, prime to $P(t)$ and
with the property that the smooth fibres of $p$ above the roots of
$G_v(t)$ have  rational points over their field of definition.

An irreducible polynomial $G(t)$ of degree $D$ defines a closed
point $M$ of degree $D$ on $\A_k^{1}$. Let $F=k(M)=k[t]/(G(t))$. Let
$\theta$ be the residue class of $t$ in $F$. We now choose the
irreducible polynomial $G(t)$ as given by
\cite[Theorem 3.1]{CTSSD98} with the field $L$ in Theorem 3.1 contains $K$, and $V'$
($V$ in Theorem 3.1) is the set of the places of $k$ at
which $L$ is totally split, such that

(i) For each place $v\in S_1$, $G(t)$ is close enough to $G_v(t)$, such
that the fibre $U_\theta$ contains an $F_w$-point $P_w$ for
each place $w$ of $F$ over $v$, and such that $\sum_{w\mid v}P_w$ is
"close" enough to $z_v$ satisfying
$$\sum_{w\mid v}inv_v(\Cor_{F_w/k_v}(\mathcal B(P_w)))=inv_v(\left<\mathcal B,z_v\right>)$$ for each $\mathcal B\in A$.

(ii) For each irreducible term $p_i(t)$ of the polynomial $P(t)$,
there exists a place $w_i$ of $F$ such that $p_i(\theta)$ is a
uniformizer at $w_i$ and  is a unit at $w$ if $w \text { is not
over } S_1\cup V'$ and $w\neq w_i$.

If $w$ is not over $S_1$, then $K.F/F$ is unramified at $w$. Then
one of $a,b,ab$ is a square in $F_w^\times$, hence the fibre $U_\theta$ contains an $F_w$-point $P_w$.

Let $\mathcal B=(x_1^2-ax_2^2,b)$. In the following we will show $inv_w(\mathcal B(P_w))=0$ for any $w$ is not over $S_1$:
\begin{enumerate}[a)]
\item Suppose $w\neq w_i$ for $1\leq i\leq m$.
\begin{enumerate}[1)]
\item If $F(\sqrt{b})/F$ is split at
$w$, then $inv_w(\mathcal B(P_w))=0$.

\item If $F_w(\sqrt{b})=F_w(\sqrt{a})$ is inert over $F_w$, then
$inv_w(\mathcal B(P_w))=0$ obviously.

\item If $F_w(\sqrt{b})/F_w$ is inert and
$F_w(\sqrt{a})/F_w$ is split, then
$$(y_1^2-by_2^2,b)_w=(z_1^2-abz_2^2,b)_w=0.$$ One has
$$inv_w(\mathcal
B(P_w))=(x_1^2-ax_2^2,b)_w=(P(\theta),b)_w=0$$ since either $P(\theta)$ is a unit at $w$  or $K.F/F$ is totally split at $w$.
\end{enumerate}
\item Suppose $w$ is some $w_i$.
For each $(p_i(t),b)\in A$, by (\ref{Ha:4.2}) and point (i), we have
$$\sum_{v\in S_1}inv_v(\Cor_{F_w/k_v}(p_i(\theta),b))=0.$$
By global reciprocity law, one has
\begin{equation}\label{equ:cycle-2}
\sum_{v\not\in S_1}inv_v(\Cor_{F_w/k_v}(p_i(\theta),b))=0.
\end{equation}
However, $inv_{w}(p_i(\theta),b)=0$ for $w$ not over $S_1$ and
$w\neq w_i$, since either $p_i(\theta)$ is a unit at $w$  or $K.F/F$ is totally split at $w$ by point (ii)
above. Therefore (\ref{equ:cycle-2}) implies
$$inv_{w_i}(p_i(\theta),b)=0.$$
Since $p_i(\theta)$ is a uniformizer at $w_i$, we have $F(\sqrt{b})/F$ is
totally split at $w_i$, hence $inv_{w_i}(\mathcal B(P_{w_i}))=0$.
\end{enumerate}

By the above arguments,  we have
$$\begin{aligned}\sum_{w\in \Omega_F}inv_w(\mathcal B(P_w))=\sum_{w\mid v\in S_1}inv_w(\mathcal B(P_w))
=\sum_{v\in S_1}inv_{v}(\left<\mathcal B,z_{v}\right>)=0
\end{aligned}$$
by (\ref{Ha:4.2}) for the last equality.
Since $F\cap K=k$, we deduce that $B$ generates the unramified
Brauer group  $\Br_{ur}(U_\theta)/\Br_0(U)$ by \cite[Theorem 4.1]{CT11}. Therefore there is no Brauer-Manin
obstruction on $U_\theta$ for $(P_w)_{w\in \Omega_F}$. Hence $U_\theta/F$ possesses
an $F$-point by \cite[Theorem 8.12]{San}.
\end{proof}

For all primes $p$ (do not need the condition (\ref{assumption:P})), we can prove the following
result (which corresponds to Theorem \ref{rational-3}).
\begin{thm}\label{cycle-3}
Let $k$ be a number field and $P(t)$ a nonzero polynomial over $k$. Let $p$
be a prime. Let $K/k$ be of degree $p$ and not cyclic. Let $V$ be
the smooth locus of the affine $k$-variety defined by
$$N_{K/k}(\Xi)=P(t).$$ If there is no Brauer-Manin
obstruction to the existence of a zero-cycle of degree 1 on a smooth
proper model $V^c$ of $V$, then there is a zero-cycle of degree 1 on
$V^c$ (defined over $k$).
\end{thm}
\begin{proof} Obviously $V$ has a closed point $N_0$ with
$k(N_0)=K$. Then we only need to show that there is a field $F/k$
with degree $[F:k]$ prime to $p$ such that $V^c(F)\neq \emptyset$.

Let $K^{cl}$ be the Galois closure of $K/k$ with the Galois group
$G$. Then $G$ is a subgroup of the symmetric group
$\mathcal S_p$. Let $H$ be the $p$-Sylow subgroup of $G$. We can see $H$ is
cyclic and of order $p$. Let $\Theta$ be the fixed field by $H$.
Then $[\Theta:k]$ is relative prime to $p$.

One has a family $\{z_{v}\}_v$ of zero-cycles of degree 1  which is
orthogonal to $\Br(V^c)$. Then one pushes $\{z_{v}\}_v$ to
$V_\Theta^c$, one gets a family of zero-cycles of degree 1  on
$V^c_{\Theta}$. A projection formula for the Brauer pairing shows
the family of zero-cycles on $V^c_{\Theta}$ is orthogonal to
$\Br(V^c_{\Theta})$. Since $K^{cl}/\Theta=K.\Theta/\Theta$ is cyclic
(of degree $p$), one gets a zero-cycle of degree 1 on $V^c_{\Theta}$
by \cite[Theorem 4.1 ]{CTSSD98}. So
there is a field $F/\Theta$ with the degree $[F:\Theta]$ prime to
$p$ such that $V^c(F)\neq \emptyset$. Since $[\Theta:k]$ is prime to
$p$, we have $[F:k]=[F:\Theta]\cdot [\Theta:k]$ is also prime to
$p$.
\end{proof}

\bf{Acknowledgment} \it{The author is grateful to the referee for a careful reading of the manuscript
and for useful remarks which improved the original presentation. The author would like to thank Professor
Colliot-Th\'el\`ene for helpful discussions and valuable
suggestions. The work is supported by National
Key Basic Research Program of China (Grant No. 2013CB834202) and
National Natural Science Foundation of China (Grant Nos. 11371210 and
11321101), and grant DE1646/2-1 of the Deutsche Forschungsgemeinschaft.}


\end{document}